\documentclass{daj}

\dajAUTHORdetails{%
  title = {Biased Halfspaces, Noise Sensitivity, and Local Chernoff Inequalities}, 
  author = {Nathan Keller and Ohad Klein},
  plaintextauthor = {Nathan Keller, Ohad Klein},
}   

\dajEDITORdetails{%
   year={2019},
   number={13},
   received={6 November 2017},   
   published={23 September 2019},  
   doi={10.19086/da.10234},       
}   

\usepackage{mathtools}
\usepackage{amsfonts}
\usepackage{amsmath}
\usepackage{amssymb}
\usepackage{amsthm}
\usepackage{bbm}

\newtheorem{theorem}{Theorem}[section]
\newtheorem{proposition}[theorem]{Proposition}
\newtheorem{lemma}[theorem]{Lemma}
\newtheorem{claim}[theorem]{Claim}
\newtheorem{corollary}[theorem]{Corollary}

\newtheorem{definition}[theorem]{Definition}

\theoremstyle{remark}\newtheorem*{remark}{Remark}

\newtheorem*{proposition*}{Proposition}
\newtheorem*{lemma*}{Lemma}
\theoremstyle{definition}\newtheorem{example}[theorem]{Example}

\newcommand{\mrm}[1]{\mathrm{#1}}

\newcommand{\bem}[0]{\begin{eqnarray*}}
\newcommand{\enm}[0]{\end{eqnarray*}}

\newcommand{\pintegers}{\mathbb{N}}

\newcommand{\reals}{\mathbb{R}}
\newcommand{\preals}{\reals_{\geq0}}
\newcommand{\spm}{\left\{ -1,1\right\}  }
\newcommand{\zo}{\left\{  0,1\right\}  }
\newcommand{\nicefrac}[2]{#1\left/#2\right.}

\newcommand{\sub}[0]{\subseteq}
\newcommand{\sm}[0]{\setminus}
\newcommand{\comp}{\circ}
\newcommand{\es}{\emptyset}
\newcommand{\ddd}{\doteqdot}

\newcommand{\one}{\mathbbm{1}}

\newcommand{\eps}{\epsilon}
\newcommand{\sgn}{\mrm{sgn}}

\newcommand{\pr}{\Pr}

\newcommand{\normts}[1]{\lbn #1 \rbn_2}
\newcommand{\normt}[1]{\normts{#1}^{2}}

\DeclareMathOperator*{\be}{\mathbb{E}}
\DeclareMathOperator*{\var}{\mrm{Var}}
\DeclareMathOperator*{\cov}{\mrm{Cov}}

\newcommand{\andd}{\wedge}
\newcommand{\Andd}{\bigwedge}
\newcommand{\orr}{\vee}
\newcommand{\Orr}{\bigvee}
\newcommand{\xor}{\oplus}

\newcommand{\lfunc}[1]{#1:\spm^{n}\to\zo}

\newcommand{\maj}{\mrm{Maj}}
\newcommand{\ii}{I}
\newcommand{\tr}{T_{\rho}}
\newcommand{\wk}[1]{W^{#1}}
\newcommand{\wo}{\wk{1}}

\newcommand{\wh}[1]{\widehat{#1}}
\newcommand{\hf}{\widehat{f}}
\newcommand{\vb}[2]{\mrm{VB}_{#2}\lbr #1\rbr}
\newcommand{\sss}{\mathbb{S}_\rho}

\newcommand{\pf}[1]{\mrm{PF}_{#1}}

\newcommand{\sff}[1]{\mrm{SF}_{#1}}
\newcommand{\scf}{\mrm{SCF}}
\newcommand{\iscf}{\mrm{ISCF}}

\newcommand{\lbr}{\left(}
\newcommand{\rbr}{\right)}
\newcommand{\lbs}{\left[}
\newcommand{\rbs}{\right]}
\newcommand{\lbc}{\left\{}
\newcommand{\rbc}{\right\}}
\newcommand{\lba}{\left|}
\newcommand{\rba}{\right|}
\newcommand{\lbn}{\left\Vert}
\newcommand{\rbn}{\right\Vert}
\newcommand{\mn}[0]{\medskip\noindent}

\begin{document}

\begin{frontmatter}[classification=text]


\author[nk]{Nathan Keller\thanks{Department of Mathematics, Bar Ilan University, Ramat Gan, Israel.
		\texttt{nathan.keller27@gmail.com}. Research supported by the Israel Science Foundation (grants no.
		402/13 and 1612/17) and the Binational US-Israel Science Foundation (grant no. 2014290).}}
\author[okay]{Ohad Klein\thanks{Department of Mathematics, Bar Ilan University, Ramat Gan, Israel.
		\texttt{ohadkel@gmail.com}.}}

\begin{abstract}
	A halfspace is a function $f\colon\{-1,1\}^n \rightarrow \{0,1\}$ of the form $f(x)=\one (a\cdot x>t)$, where $\sum_i a_i^2=1$.
	
	We show that if $f$ is a halfspace with $\mathbb{E}[f]=\eps$ and $a'=\max_i |a_i|$, then the degree-1 Fourier weight of $f$ is
	$W^1(f)=\Theta(\eps^2 \log(1/\eps))$, and the maximal influence of $f$ is $I_{\max}(f)=\Theta(\eps \min(1,a' \sqrt{\log(1/\eps)}))$.
    These results, which determine the exact asymptotic order of $W^1(f)$ and $I_{\max}(f)$,
	provide sharp generalizations of theorems proved by Matulef, O'Donnell, Rubinfeld, and Servedio, and settle a
	conjecture posed by Kalai, Keller and Mossel.
	
	In addition, we present a refinement of the definition of noise sensitivity which takes into consideration the bias of the function,
	and show that (like in the unbiased case) halfspaces are noise resistant, and, in the other direction, any noise resistant function
	is well correlated with a halfspace.
	
	Our main tools are `local' forms of the classical Chernoff inequality, like the following one proved by Devroye and Lugosi (2008):
	Let $\lbc x_i\rbc$ be independent random variables uniformly distributed in $\spm$, and let $a_i\in\preals$ be such that $\sum_i a_{i}^{2}=1$.
	If for some $t\geq 0$ we have $\pr\lbs \sum_{i} a_i x_i > t\rbs=\eps$, then $\pr[\sum_{i} a_i x_i>t+\delta]\leq \frac{\eps}{2}$ holds for $\delta\leq c/\sqrt{\log(1/\eps)}$, where $c$ is a universal constant.
\end{abstract}
\end{frontmatter}

\section{Introduction}

Analysis of Boolean functions (that is, functions of the form $f\colon\{-1,1\}^n \rightarrow \{0,1\}$), was initiated about 30 years ago, and has grown into a prolific research field, with numerous applications and connections to other fields of mathematics, computer science, physics, and economics (see~\cite{O'D14}). Halfspaces (i.e., Boolean functions of the form $f(x)= \one \left(\sum_i a_i x_i > t \right)$) have always been a central object of study in the field; noise sensitivity (which studies the effect of small perturbations of the input on the function output) joined in 1999, bringing thrilling applications to percolation theory. As was shown by Benjamini, Kalai, and Schramm~\cite{BKS99}, halfspaces and noise sensitivity are closely related, and we further explore the relation in this paper.

Usually, the Boolean functions of interest are \emph{unbiased}, meaning that they satisfy the condition $\mathbb{E}[f]=1/2$. As a result, many of the central notions and results in analysis of Boolean functions assume that the function is (roughly) unbiased. Noise sensitivity is a notable example. However, in various applications the effect of the bias is central (e.g., threshold phenomena~\cite{KK06}, correlation inequalities~\cite{Talagrand96}, isoperimetry~\cite{EKL16}, and social choice theory~\cite{Kel12}), and therefore getting rid of the assumption on $\mathbb{E}[f]$ is desirable.

In this paper, we study \emph{biased} Boolean functions, concentrating on halfspaces, noise sensitivity, and the relation between noise resistance and strong correlation with a halfspace. In particular, we determine the exact asymptotic order of the first-degree Fourier weight and the maximal influence of halfspaces, and we show that the relation between being resistant to noise and being well correlated with a halfspace carries over from the unbiased case to biased functions, under appropriate definitions. Our techniques are somewhat non-standard for the types of questions we study: while most previous results on these problems were obtained using discrete Fourier analysis and hypercontractivity, our main tool is a local variant of the Chernoff inequality, which allows one to compare the rates of decay of the probability $\Pr[\sum a_i x_i >t]$ (where $\{x_i\}$ are independent and uniformly distributed in $\spm$), as a function of $t$.

\subsection{First-degree Fourier weight and maximal influence of halfspaces}

A \emph{halfspace}, or a \emph{Linear Threshold Function (LTF)}, is a Boolean function $f(x)= \one \left(\sum_i a_i x_i > t \right)$, where $a \in \mathbb{R}^n$ and $t \in \mathbb{R}$. (The vector $a$ is usually normalized such that $\sum_i a_i^2=1$). In the last half-century, halfspaces have been a central object of study in various areas, such as complexity theory, optimization, machine learning, and social choice theory (see, e.g.,~\cite{GHR92,Hu65,MP68,O'D14,STC00,TZ92,Yao90}).

\subsubsection{First-degree Fourier weight of halfspaces}

A major tool frequently used in the study of halfspaces is their \emph{Fourier expansion} -- namely, their unique representation as a multilinear polynomial: $f = \sum_{S \subseteq \{1,2,\ldots,n\}} \hat f(S) x^S$, where $x^S \ddd  \prod_{i \in S} x_i$. Of special importance here are the \emph{first-degree} (or first-level) \emph{Fourier coefficients} -- the coefficients $\hat f(S)$ which correspond to singletons $S=\{i\}$, as Chow~\cite{Cho61} proved in 1961 that a halfspace is uniquely determined by the set of its first-degree coefficients (together with $\hf(\es)$).

By Parseval's identity, the \emph{total Fourier weight} of a Boolean function, $\sum_S \hat f(S)^2$, is equal to $\mathbb{E}[f]$. It is well known that most of the Fourier weight of \emph{unbiased} halfspaces (i.e., halfspaces of the form $f=\one(\sum_i a_i x_i > 0)$, whose expectation is $1/2$) is concentrated on the first degree. More precisely, Gotsman and Linial~\cite{GL94} proved that for any unbiased halfspace $f$, we have $W^1(f) \ddd \sum_{|S|=1} \hat f(S)^2 \geq 1/8$. (The best currently known bound is $1/8+c$ for some explicit $c>0$~\cite{DDS16}, and it is conjectured that the `correct' bound is $1/2\pi$, which is asymptotically attained by the majority function $f(x)=\one (\sum_i \frac{1}{\sqrt{n}}x_i>0)$.) Hence, the first-degree Fourier weight of unbiased halfspaces is within a constant multiplicative factor of the maximal possible weight.

A question that naturally arises is whether a similar phenomenon holds for \emph{biased} halfspaces. Here, the bounds must depend on the bias of the function, as the \emph{level-1 inequality}~\cite{Chang02,IMR14,Talagrand96} asserts that for any Boolean function $f$, we have
$W^1(f) \leq 2 \mathbb{E}[f]^2 \log(1/\mathbb{E}[f])$. (Note that this improves significantly over the bound $\mathbb{E}[f](1-\mathbb{E}[f])$ that follows from merely applying Parseval's identity). In view of the results for unbiased halfspaces, it makes sense to conjecture that any halfspace $f$ satisfies
\begin{equation}\label{Eq:Half1}
	W^1(f) \geq c \mathbb{E}[f]^2 \log(1/\mathbb{E}[f]),
\end{equation}
where $c$ is a universal constant. Matulef, O'Donnell, Rubinfeld, and Servedio~\cite[Theorem~48]{MORS10} showed that~\eqref{Eq:Half1}, and actually a more precise bound, holds for halfspaces all of whose coefficients $a_i$ are sufficiently small, called \emph{low-influence halfspaces}. This result plays a crucial role in the algorithm of~\cite{MORS10} for testing halfspaces, and in the algorithm of O'Donnell and Servedio~\cite{OS11} for learning halfspaces. Kalai, Keller, and Mossel~\cite[Open Problem~6.2]{KKM16} asked to determine all functions for which~\eqref{Eq:Half1} holds (i.e., all functions for which the level-1 inequality is tight up to a constant factor), and conjectured that~\eqref{Eq:Half1} holds for all halfspaces. We prove this conjecture, using a local Chernoff inequality (to be presented in the sequel).
\begin{theorem}\label{Thm:Main-W^1}
	There exists a universal constant $c>0$ such that for any halfspace $f=\one(\sum_i a_i x_i>t)$ with $\be[f]\leq \frac{1}{2}$, we have
	\[
	c \mathbb{E}[f]^2 \log \frac{1}{\mathbb{E}[f]} \leq W^1(f) \leq 2 \mathbb{E}[f]^2 \log \frac{1}{\mathbb{E}[f]}.
	\]
\end{theorem}
Using the technique of~\cite[Proposition~5.3]{KKM16}, Theorem~\ref{Thm:Main-W^1} provides a large class of tightness examples for a well-known correlation inequality of Talagrand~\cite{Talagrand96} which asserts that for any two monotone Boolean functions $f,g$, we have $\mathrm{Cov}(f,g) \geq c \varphi(\sum_{i} \hat f(\{i\}) \hat g(\{i\}))$, where $\varphi(x)=x/\log(e/x)$, and $c$ is a universal constant.
\begin{corollary}
	Let $f=\one(\sum_i a_i x_i >t)$ be a halfspace, and let $g = \one(\sum_i a_i x_i \geq -t)$ be the dual halfspace. Then the pair $(f,g)$ is a tightness example for Talagrand's inequality, meaning that $\mathrm{Cov}(f,g) = \Theta \left(\varphi(\sum_{i} \hat f(\{i\}) \hat g(\{i\})) \right)$, where
	$\varphi(x)=x/\log(e/x)$.
\end{corollary}
We generalize Theorem~\ref{Thm:Main-W^1} to the $k$th-degree Fourier weight of halfspaces, showing that the \emph{level-$k$ inequality} (see~\cite[Chapter~9]{O'D14}) is tight (up to a factor that depends only on $k$) for all low-influence halfpaces. In particular, we prove:
\begin{theorem}\label{Thm:Main-W^k}
	For any $k$, there exist constants $c_1,c_2,c_3,c_{4}$ depending only on $k$, such that for any halfspace $f=\one \lbc \sum_i a_i x_i > t \rbc$ with $\mathbb{E}[f] \leq c_1$ and $\ii_{\max}(f)\leq c_{4} \mathbb{E}[f]$, we have
	\[
	c_2 \mathbb{E}[f]^2 \lbr\log \frac{1}{\mathbb{E}[f]}\rbr^k \leq W^k(f) \leq c_3 \mathbb{E}[f]^2 \lbr\log \frac{1}{\mathbb{E}[f]}\rbr^k.
	\]
\end{theorem}

\subsubsection{The maximal influence of halfspaces}

The influence of the $k$th coordinate on a Boolean function $f$ is defined as
\[
I_k(f) \ddd \Pr_{x\sim \spm^{n}}[f(x) \neq f(x \oplus e_k)],
\]
where $x \oplus e_k$ is obtained from $x$ by flipping the $k$th coordinate. The \emph{total influence} of $f$ is $I(f) \ddd \sum_k I_k(f)$.

Influences have been studied very extensively in the last decades, and their applications span a wide variety of fields, including percolation theory~\cite{BKS99}, social choice theory~\cite{Kalai02,MOO10}, hardness of approximation~\cite{DS05,Hastad01}, correlation inequalities~\cite{KKM16,Talagrand96}, etc. (see the survey~\cite{KS06}).

At first sight, it may seem that the $k$th influence of a halfspace $\one(\sum_i a_i x_i >t)$ is `proportional' to the weight $a_k$. However, this is not the case; for example, the halfspace $f=\one(\frac{4}{5}x_1+ \frac{3}{5}x_2>0)$ is equal to the dictator function $f(x)=\one(x_1>0)$, and the influence of the second coordinate on it is zero. Hence, it is desirable to find a relation between the influences and the weights, to the extent that such a relation exists.

In~\cite[Theorem~36]{MORS10}, Matulef et al. proved a lower bound on the maximal influence of halfspaces:
\[
\max_k \{I_k(f)\} \geq \max_i \{a_i\} \cdot \mathbb{E}[f]^6 \log(1/\mathbb{E}[f]),
\]
and used it as another central component in their algorithm for testing halfspaces. The authors of~\cite{MORS10} conjectured that the lower bound can be improved to $\Omega(\max_i \{a_i\} \mathbb{E}[f])$. This conjecture was later proved by Dzindzalieta and G\"{o}tze~\cite{DG13}.

We determine the exact asymptotic order of the largest influence of a halfspace:
\begin{theorem}\label{Thm:Main-influence}
	There exist universal constants $c_1,c_2$ such that for any halfspace $f= \one(\sum_i a_i x_i > t)$ with $\be[f]\leq \frac{1}{2}$ and $a_1 \geq a_2 \geq \ldots \geq a_n \geq 0$, we have
	\[
	c_1 \mathbb{E}[f] \min \{1,a_1 \sqrt{\log(1/\mathbb{E}[f])}\} \leq \max_{i} \ii_i(f) \leq c_2 \mathbb{E}[f] \min \{1,a_1 \sqrt{\log(1/\mathbb{E}[f])}\}.
	\]
\end{theorem}
In view of the aforementioned example, even the fact that there at all exists a fixed relation between the maximal influence of a halfspace and its largest weight is perhaps somewhat surprising.

\subsubsection{The vertex boundary of halfspaces}

A halfspace $f$ naturally corresponds to the set $\one_f = \{x:f(x)=1\}$, which may be viewed as a subset of the discrete cube graph. A natural isoperimetric question one may ask is: what is the relation between the size of this set (which is, of course, $2^n \cdot \mathbb{E}[f]$), and the size of its \emph{boundary}?
In finite graphs, there are two classical types of boundary of a set $S$: the \emph{edge boundary}, which consists of the edges that connect a vertex in $S$ with a vertex in the complement of $S$, and the \emph{vertex boundary}, which consists of the vertices in $S$ that have a neighbor outside $S$ (or, vice versa, of the vertices outside $S$ that have a neighbor in $S$).

It is easy to see that the size of the edge boundary of the set $\one_f$ is equal (up to normalization) to the \emph{total influence} $I(f) \ddd \sum_k I_k(f)$, and thus is usually easier to deal with. We show that for halfspaces, the asymptotic size of the vertex boundary $\partial(\one_f)$ admits a nice  expression in terms of $\mathbb{E}[f]$ and the maximal weight $|a_1|$.
\begin{theorem}\label{Thm:Main-boundary}
	There exist universal constants $c'_1,c'_2$ such that for any halfspace $f= \one(\sum_i a_i x_i >t)$ with $\be[f]\leq \frac{1}{2}$ and $a_1 \geq a_2 \geq \ldots \geq a_n \geq 0$, we have
	\[
	c'_1 \mathbb{E}[f] \min \{1,a_1 \sqrt{\log(1/\mathbb{E}[f])}\} \leq |\partial(\one_f)|/2^n \leq c'_2 \mathbb{E}[f] \min \{1,a_1 \sqrt{\log(1/\mathbb{E}[f])}\}.
	\]
\end{theorem}
The theorem is proved by showing that for halfspaces, the vertex boundary is approximately equal to the largest influence, and then applying Theorem~\ref{Thm:Main-influence}. We note that other relations between the measure of the vertex boundary and influences were obtained by Talagrand~\cite{Tal97}.

\subsection{Noise sensitivity of biased functions and correlation with halfspaces}

A Boolean function is called \emph{noise sensitive} if flipping each of its input bits with a small probability affects its output `significantly'. Otherwise, it is called \emph{noise resistant}. Formally, the \emph{noise stability} of a function $f$ at noise rate $1-\rho$ is defined as
\begin{equation}\label{Eq:Noise0}
\mathbb{S}_{\rho}(f) = \mathrm{Cov}(f(x),f(y)),
\end{equation}
where $y$ is obtained from $x$ by independently keeping each coordinate of $x$ unchanged with probability $\rho$,
and replacing it by a random value with probability $1-\rho$. A sequence of functions $\{f_m:\{-1,1\}^{n_m} \rightarrow \{0,1\}\}$ is called \emph{asymptotically noise sensitive} if for any constant $\rho\in(0,1)$, we have $\lim_{m \rightarrow \infty} \mathbb{S}_{\rho}(f_m)=0$. For the sake of simplicity, we consider a single function $f$ and say that it is noise sensitive if $\mathbb{S}_{\rho}(f)=o_n(1)$, and is noise resistant otherwise.

Noise sensitivity is a fundamental property of Boolean functions that has been studied extensively over the last two decades. Its applications span several areas, including machine learning (e.g.,~\cite{DRST14,KOS08}), hardness of approximation (e.g.,~\cite{KKMO07,MOO10}), percolation theory (e.g.,~\cite{GPS10,SS10}), and social choice theory (e.g.,~\cite{Kalai02,MOO10}).

A main result of the seminal work of Benjamini, Kalai and Schramm~\cite{BKS99} that initiated the study of noise sensitivity, is that noise resistance is closely related to strong correlation with a halfspace, and to a property of the Fourier expansion. Specifically, they showed the following result.
\begin{theorem}[\cite{BKS99}]\label{Thm:BKS-aux}
	{\rm (a)} A monotone Boolean function $f$ is noise resistant if and only if $W^1(f)=\Omega(1)$.
	
	\medskip \noindent {\rm (b)} Any unbiased halfspace is noise resistant (and actually, satisfies a stronger property called `noise stability').
	
	\medskip \noindent {\rm (c)} For any noise resistant monotone Boolean function $f$, there exists an unbiased halfspace $g$ such that $\mathrm{Cov}(f,g) = \Omega(1)$.
\end{theorem}
We note that in the non-monotone case the situation is more complex. Indeed, as was shown recently by Mossel and Neeman~\cite{MN16}, even the stronger assumption that $f$ is noise stable is not sufficient for guaranteeing the existence of a halfspace $g$ such that $\mathrm{Cov}(f,g)=\Omega(1)$.

\medskip

The definition of noise sensitivity is `not interesting' for highly biased functions (i.e., when $\mathbb{E}[f]$ is close to $0$ or to $1$), as any such function is clearly noise sensitive. Hence, it is natural to ask what should be the `right' definition of noise sensitivity for highly biased functions. Inspired by Theorem~\ref{Thm:BKS-aux}, we propose a `Fourier-theoretic' definition.

Note that Theorem~\ref{Thm:BKS-aux}(a) asserts that an unbiased monotone function is noise resistant if and only if its first-degree Fourier weight is, up to a constant factor, the maximum possible. For general functions, the aforementioned `level-1 inequality' asserts that $W^1(f) = O(\mathbb{E}[f]^2 \log(1/\mathbb{E}[f]))$. Based on this, we say that $f$ is noise resistant if $W^1(f)$ is within a constant factor of the maximum possible. Formally:
\begin{definition}
	A function $f\colon\{-1,1\}^n \rightarrow \{0,1\}$ is called \emph{Fourier noise resistant} if
\[
W^1(f) \geq c \mathbb{E}[f]^2 \log(1/\mathbb{E}[f]),
\]
for some universal constant $c$.
\end{definition}

Theorem~\ref{Thm:Main-W^1} allows us to claim that with respect to this definition, the close relation between noise resistance and strong correlation with a halfspace holds also for biased functions. Indeed, one direction (i.e., that any halfspace is Fourier noise resistant) is exactly the assertion of Theorem~\ref{Thm:Main-W^1}. In the converse direction, Mossel and Neeman~\cite[Proposition~3.2]{MN16} showed that for any Boolean function $f$, there exists a halfspace $g$ such that $\mathrm{Cov}(f,g) \geq \Omega(\wo(f) / \sqrt{\mathbb{E}[f](1-\mathbb{E}[f])})$. We show the following sharp bound, which is always stronger than the bound of~\cite{MN16} by the level-1 inequality.
\begin{theorem}\label{Thm:Main-Cor}
	For any Boolean function $f$, there exists a halfspace $g$ such that
	\begin{equation}\label{Eq:Intro-Cor1}
	\mathrm{Cov}(f,g) \geq c\sqrt{\wo(f) / \log(e/\wo(f))},
	\end{equation}
	where $c$ is an absolute constant. In particular, if $f$ is Fourier noise resistant and $\mathbb{E}[f] \leq 1/2$ then there exists a halfspace $g$ such that $\mathrm{Cov}(f,g) = \Omega(\mathbb{E}[f])$.
\end{theorem}
Note that the correlation asserted in the theorem is clearly within a constant factor of the maximum possible, as $\mathrm{Cov}(f,g) \leq \mathbb{E}[f]$ for any $f,g$. An interesting feature of Theorem~\ref{Thm:Main-Cor} is that unlike the classical result of~\cite{BKS99}, the strong correlation with a halfspace is guaranteed even if the function $f$ is not monotone. This is somewhat surprising, as most known correlation bounds (such as FKG-type inequalities~\cite{FKG71}) hold only for monotone functions.

\medskip

Finally, we show that for monotone functions, strong correlation with a halfspace is implied also by a `probabilistic' notion of noise resistance. Here, the rate of noise we consider is $1-c/\log(1/\mathbb{E}[f])$, for a fixed `small' constant $c$ (i.e., $\rho=c/\log(1/\mathbb{E}[f])$). It is easy to show (see Section~\ref{sec:noise}) that for this noise rate, any function $f$ satisfies $\mathbb{S}_\rho(f) = O(\mathbb{E}[f]^2)$. Recalling that the classical definition of noise resistance is $\mathbb{S}_\rho(f) = \Omega(1)$, which is within a constant factor of the maximal possible value, a natural definition of noise resistance in our setting is the requirement $\mathbb{S}_\rho(f) = \Omega(\mathbb{E}[f]^2)$.
\begin{proposition}\label{Prop:Prob-noise}
	There exists a universal constant $c > 0$ such that for any monotone function $f\colon \spm^n \rightarrow \zo$ with $\mathbb{E}[f] \leq \frac{1}{2}$, if $\mathbb{S}_{c/\log(1/\mathbb{E}[f])}(f) = \Omega(\mathbb{E}[f]^2)$, then $W^1(f) =\Omega(\mathbb{E}[f]^2 \log(1/\mathbb{E}[f]))$, and consequently, there exists a halfspace $g$ such that $\mathrm{Cov}(f,g) = \Omega(\mathbb{E}[f])$.
\end{proposition}

\subsection{Local Chernoff Inequalities}

Tail estimates for weighted sums of independent random variables are among the most frequently used probabilistic tools in combinatorics and
theoretical computer science. A standard example is Hoeffding's inequality which asserts that if $\{x_i\}_{i=1}^n$ are independent mean-zero random
variables with $\forall i\colon |x_i| \leq 1$ and $\{a_i\}_{i=1}^n$ are real numbers that satisfy $\sum_i a_i^2 \leq 1$, then for any $t>0$,
\[
\Pr \left[\sum_i a_i x_i > t \right] \leq \exp(-t^2/2).
\]
In the commonly-studied case where each $x_i$ is uniformly distributed in $\{-1,1\}$ (also called Rademacher random variables), stronger bounds
can be obtained, which essentially state that $\sum a_i x_i$ is distributed `like' a Gaussian random variable. In particular, there exists
a constant $c$ such that for any $t>0$,
\begin{equation}\label{Eq:Intro-tail1}
\Pr \left[\sum_i a_i x_i > t \right] \leq c \Pr[Z>t],
\end{equation}
where $Z \sim N(0,1)$. (This is a result of Eaton~\cite{Eaton74};
the `correct' value of $c$ was recently determined by Bentkus and Dzindzalieta~\cite{BD15} to be $\approx 3.178$.) This phenomenon is also demonstrated by the Central Limit Theorem, or its more quantitative form, the Berry-Esseen Theorem (see, e.g.,~\cite{Fel68}), which implies that for any interval $I$,
\begin{equation}\label{Eq:Intro-tail2}
\left|\Pr \left[\sum_i a_i x_i \in I \right] - \Pr [Z \in I] \right| \leq c' \sum_{i} |a_{i}|^{3} \leq c' \max_i \{|a_i|\},
\end{equation}
where $c'$ is an absolute constant. (The claim holds, e.g., for $c'=1$; the best currently known bound on $c'$ was obtained by Shevtsova~\cite{She13}).
The `local Chernoff inequalities' we consider in this paper assert that the \emph{rate of decay} of $\pr[\sum a_i x_i > t]$ as a function of $t$ is also essentially equal to that of a Gaussian random variable $Z \sim N(0,1)$.

\subsubsection{A local Chernoff inequality of Devroye and Lugosi, via a general method of Benjamini, Kalai, and Schramm}

In a remark ending their seminal paper on the variance of first passage percolation~\cite{BKS03}, Benjamini et al. suggested a general method for deriving `local' tail estimates for random variables from hypercontractive inequalities. Essentially, in order to obtain a local tail estimate for $f$, one considers the function $g_t(x)=\max(f(x),t)$, where $t$ is chosen such that $\pr[f>t]=\eps$. Then one uses a theorem of Talagrand~\cite{Tal94} (Theorem~\ref{Thm:Tal} below, whose proof relies on hypercontractivity) to show that $\var(g_t)$ is `small' (as a function of $\eps$), and deduces an upper bound on the minimal $\delta$ such that $\pr(f>t+\delta)\leq \eps/2$ using Chebyshev's inequality.

In~\cite{DL08}, Devroye and Lugosi developed the method of~\cite{BKS03} and used it to obtain various local tail bounds. In particular, applying the method to the function $f=\sum a_i x_i$, they proved the following tail estimate:
\begin{theorem}[Devroye and Lugosi, 2008]\label{thm:intro-strong_chernoff}
	Let $\lbc x_i\rbc$ be independent random variables uniformly distributed in $\spm$, and let $a_i\in\preals$ be such that $\sum_i a_{i}^{2}=1$. There exists a universal constant $c>0$ such that if $t\geq 0$ and $\eps=\pr\lbs \sum_{i} a_i x_i > t\rbs$, then $\pr[\sum_{i} a_i x_i>t+\delta]\leq \frac{\eps}{2}$, for $\delta\leq \frac{c}{\sqrt{\log(1/\eps)}}$.
\end{theorem}
This shows that the `relative' decay of the tail probability $\pr\lbs \sum_{i} a_i x_i > t\rbs$ is essentially equal to that of a Gaussian random variable $Z$. Indeed, an easy computation yields that if for some $t \geq 0$ we have $\Pr[Z>t] = \eps$, then the minimal $\delta$ such that $\Pr[Z>t+\delta] \leq \epsilon/2$ is of order $\Theta\lbr 1/\sqrt{\log(1/\eps)}\rbr$.

Theorem~\ref{thm:intro-strong_chernoff} implies that if for some $t>0$, the probability $\Pr[\sum_i a_i x_i >t]$ is much smaller than the Gaussian-like bound provided by~\eqref{Eq:Intro-tail1}, then for any $t'>t$, the probability $\Pr[\sum_i a_i x_i >t']$ will `remain' much smaller than that of a Gaussian random variable. The theorem is tight up to a constant factor, e.g., for $X=\sum_{i=1}^n \frac{1}{\sqrt{n}}x_i$ where $n$ is sufficiently large; this follows immediately from~\eqref{Eq:Intro-tail2}, using the exact rate of decay of the Gaussian distribution.

Following the notation of~\cite{DL08} where such estimates were called `local tail bounds', we refer to Theorem~\ref{thm:intro-strong_chernoff} and its variants as \emph{local Chernoff inequalities}.\footnote{We note that possibly, the name `Hoeffding' should be used here instead of `Chernoff'. However, as it is quite common to call all results of this type `Chernoff-type inequalities', we prefer to use this name.}

\subsubsection{Refined variants, via log-concavity}

For our applications, we will need a refined inequality, which takes into consideration the weights $a_i$:
\begin{theorem}\label{thm:intro-strong_chernoff2}
	Let $\lbc x_i\rbc$ be independent random variables uniformly distributed in $\spm$, let $a_i\in\preals$ be such that $\sum_i a_{i}^{2}=1$, and let $t\geq 0$. Denote $\eps=\pr\lbs \sum_{i} a_i x_i > t\rbs$, and let $\delta$ be minimal such that $\pr[\sum_{i} a_i x_i>t+\delta]\leq \frac{\eps}{2}$.
	
	If $B,S$ is any partition of $\{1,2,\ldots,n\}$ (which corresponds to `big' and `small' values of the $a_i$'s), then one of the following holds:
	\begin{itemize}
		\item $\lba B\rba\geq \frac{1}{2}\log(1/\eps)$, or
		\item $\delta\leq c\sqrt{\frac{\sum_{i\in S} a_i^2}{\log(1/\eps)}}$,
	\end{itemize}
where $c$ is a universal constant.
\end{theorem}

\noindent We also prove the following inequality, which applies in the slightly more general case of bounded symmetric random variables:
\begin{theorem}\label{thm:intro-weak_chernoff}
	Let $X=\sum x_{i}$, where $\{x_i\}$ are independent symmetric (around $0$) random variables with $\lba x_i\rba \leq a_i$ almost surely, and let $F(t)=\pr\lbs X>t\rbs$. Set $m=2\max_{i}\lbc a_{i}\rbc $, and let $c\in(0,1)$. If $\epsilon=F(t)$ for some $t\geq0$, and $\delta\geq0$ is minimal such that $F(t+\delta)\leq c\cdot F(t)$, then we have
	\[
	\delta\leq O\lbr m+\log\lbr 2/c\rbr \sqrt{\frac{\sum_{i}a_{i}^{2}}{\log(1/\epsilon)}}\rbr.
	\]
\end{theorem}

The main tool in the proof of Theorem~\ref{thm:intro-weak_chernoff}, which we also use to present an alternative proof of Theorem~\ref{thm:intro-strong_chernoff}, is a `relaxed log-concavity' lemma:
\begin{lemma}\label{lem:intro-log_concavity}
	Let $X= \sum_i x_i$ be a sum of independent real random variables, and denote $F(t) = \pr[X>t]$. Set $m = \max_{i\in[n]}\{\sup {x_i}-\inf {x_i}\}$. For any $b\leq c\leq d$, we have
	\[
	F(d)F(b)\leq F(c)F(b+d-c-m).
	\]
\end{lemma}
We prove the lemma by constructing an explicit measure-preserving injection from the event $\{X_1 > d\}\times\{X_2 > b\}$ to the event $\{X_1 > c\}\times\{X_2 > b+d-c-m\}$, where $X_1,X_2$ are two identical, independent, copies of $X$. The idea is to swap an appropriate fraction of $X_1$ and $X_2$, in a way that increases $X_2$, at the expense of decreasing $X_1$. Theorem~\ref{thm:intro-strong_chernoff2} follows from Lemma~\ref{lem:intro-log_concavity} and Hoeffding's inequality via some technical computations.

We believe that these `local' tail estimates and their variants, as well as the log-concavity lemma, will be useful in other contexts as well.

\subsection{Organization of the paper}

This paper is organized as follows. In Section~\ref{sec:conventions} we present notation and conventions to be used throughout the paper. In Section~\ref{sec:concentration} we prove Lemma~\ref{lem:intro-log_concavity} and another concentration lemma, and in Section~\ref{sec:Chernoff} we use these lemmas to prove the local Chernoff inequalities (namely, Theorems~\ref{thm:intro-strong_chernoff},~\ref{thm:intro-strong_chernoff2} and~\ref{thm:intro-weak_chernoff}). Our results on the first-degree Fourier weight (Theorem~\ref{Thm:Main-W^1}), the maximal influence (Theorem~\ref{Thm:Main-influence}), the vertex boundary size (Theorem~\ref{Thm:Main-boundary}), and the $k$th-degree Fourier weight (Theorem~\ref{Thm:Main-W^k}) of halfspaces are presented in Sections~\ref{sec:First-Degree},~\ref{sec:lower_bound_I1},~\ref{sec:boundary}, and~\ref{sec:W^k}, respectively. Finally, we study noise sensitivity of biased functions and prove Theorem~\ref{Thm:Main-Cor} and Proposition~\ref{Prop:Prob-noise} in Section~\ref{sec:noise}.


\section{Conventions}\label{sec:conventions}

In this section we present notation and conventions that we will use throughout the paper.

\mn\textbf{1.} For a Boolean function $f$, we denote $\mu(f) \ddd \pr[f=1] \ddd \be_{x\sim \spm^{n}}[f]$. We say that $f$ is \emph{almost unbiased} if $c \leq \mathbb{E}[f] \leq 1-c$ for a universal constant $c$ (whose exact value does not matter). On the other hand, $f$ is said to be \emph{strongly biased} if $\mathbb{E}[f]=o(1)$ or $\mathbb{E}[f]=1-o(1)$. Notice these notions formally make sense only for families of Boolean functions. However, as we use them only for enhancing intuition, we skip over the accurate formulation when the meaning is clear from the context.

\mn\textbf{2.} A halfspace is a Boolean function of the form ${f=\one\{a\cdot x > t\}}$ with $a\in \reals^{n}$ and $t\in \reals$.
We are always going to assume that $a_{i} \in \preals$ and that $a_{1} \geq \ldots \geq a_{n} > 0$. Furthermore, we frequently assume $\mu(f)\leq \frac{1}{2}$; this mostly does not affect generality since we can alternatively investigate the dual function $g(x)=1-f(-x)$ which shares many properties with $f$ (note that $g$ is also a halfspace).

\mn\textbf{3.} We sometimes identify the halfspace $f$ with $2f-1=\sgn(a\cdot x - t)$, where $\sgn$ is the sign function, and we choose $\sgn(0)=-1$. Furthermore, since there are only finitely many ($2^{n}$) values for $a\cdot x$, we may increase $t$ a little without changing $f$. Moreover, notice that as long as we are interested in a particular halfspace $f=\one\{a\cdot x > t\}$, we may assume that $a\cdot x$ does not assume any finite set of values since slightly altering $a$ does not change $f$.

\mn\textbf{4.} For an $a\in \preals^{n}$ as above and an $s\in \reals$ we write $f_{s}(x)=\one\{a\cdot x> s\}$. (Notice that the notation $f_{s}(x)$ is used only for halfspaces, where $s$ always denotes the threshold). Additionally, we use the notation $F(t)=\pr_{x}[ a\cdot x > t ]$, and so, $F(t)=\mu(f_t)$. Furthermore, we regularly write $\eps \ddd \mu(f_t) \ddd F(t)$.

\mn\textbf{5.} The letters $\beta, \gamma, \delta$ are usually used to describe a significant decay of $F$; e.g., in several places $\beta$ is chosen to be the minimal positive real value satisfying $F(t+\beta) \leq \frac{1}{3} F(t)$.

\mn\textbf{6.} When we write $\lg(x)$ we always mean $\log_{2}(x)$; $\log(x)$ always denotes $\ln(x)=\log_{e}(x)$.

\mn\textbf{7.} $U(0, t)$ is the uniform real distribution over the interval $(0, t)$.


\section{Two Concentration Lemmas}
\label{sec:concentration}

In this section we prove two concentration results concerning sums of independent random variables. These lemmas are central
tools in the proof of the local Chernoff inequalities in Section~\ref{sec:Chernoff}, and are also used in other proofs in the
sequel.

The first is Lemma~\ref{lem:intro-log_concavity}, which asserts that if $X=\sum_{i\in[n]}x_i$, where $\lbc x_i \rbc_{i\in[n]}$ are independent real random variables, and we denote $F(t) = \pr[X>t]$ and $m = \max_{i\in[n]}\{\sup {x_i}-\inf {x_i}\}$, then for any $b\leq c\leq d$, we have
	\begin{equation}\label{eq:log_concavity_statement}
	F(d)F(b)\leq F(c)F(b+d-c-m).
	\end{equation}

Notice the $(-m)$ in Equation~\eqref{eq:log_concavity_statement} can not in general be omitted. This is because $b,c,d$ might not be `aligned' to values achievable by $X$. To compare, in the context of the usual notion of discrete log-concave distribution, the underlying random variable assumes only integer values. This does not capture the behavior of the variables $X$ discussed in Lemma~\ref{lem:intro-log_concavity}.

The second result is a concentration lemma which assumes (in addition) that the random variables are symmetric.
\begin{lemma}\label{lem:decay_of_intervals}
	Let $x_i$ be independent symmetric (around $0$) random variables with $\lba x_i\rba\leq a_i$, almost surely. For any $m\geq\max_i\{a_i\}$ and
any $s,t$ such that $0\leq s\leq t$, we have
	\begin{equation}\label{eq:decay_of_intervals}
	\pr \lbs \sum_{i=1}^{n} x_i \in (t-m, t+m]\rbs \leq 5\pr \lbs \sum_{i=1}^{n} x_i \in (s-m, s+m]\rbs.
	\end{equation}
\end{lemma}
Note that while it may seem that the lemma `should' hold with the constant in the right hand side of~\eqref{eq:decay_of_intervals} equal to $1$, we show below
by an explicit example that this constant must be at least $2$.

\medskip

We prove both lemmas constructively using \emph{injective measure preserving maps} from the set of events that represent the l.h.s. into
the set of events that represent the r.h.s. We introduce several injective transformations which we will use to construct the maps in Section~\ref{sec:sub:maps} and present the proof of the lemmas in Section~\ref{sec:sub:lemmas-proof}.

\subsection{Auxiliary injective transformations}
\label{sec:sub:maps}

\begin{definition}[Prefix/suffix flip]\label{def:prefix_flip}
    Let $r$ be a real number, and let $u,v \in \reals^n$  be two vectors whose `partial sums of differences' $S_k(u,v)  \ddd  \sum_{i=1}^{k} \lbr u_i-v_i \rbr$
    satisfy $\max_{k\in[n]}\{S_k(u,v)\}\geq r$. We define $\sff{r}(u,v)\in \reals^n\times \reals^n$ as follows.
	Take $t=\min \lbc i\in\lbc 0,\ldots,n\rbc \mid S_{i}(u,v)\geq r \rbc$ (where $S_0(u,v) \ddd 0$), and set
	\[
	\forall i\in [n]\colon \sff{r}(u,v)_{i}=\begin{cases}
	(u_i,v_{i}), & i\leq t\\
	(v_i,u_{i}), & i>t
	\end{cases}.
	\]
	That is, we choose the first index $t$ in which the partial sum $S_t(u,v)$ exceeds $r$, and interchange the coordinates of the vector $(u,v)$
    in all indices later than $t$. (This is called a `suffix flip').

	For $r \in \reals$ and a single vector $u \in \reals^n$ such that $\max_{k\in[n]}\{S_k(u,-u)\}\geq r$, we define $\pf{r}(u)$ to be the unique $v\in \reals^n$ that satisfies $(-v,v) = \sff{r}(u,-u)$.

    That is, we choose the first index $t$ in which the partial sum $\sum_{i \leq t} u_i$ exceeds $r/2$ and flip all coordinates of $u$ with indices
    no later than $t$. (This is called a `prefix flip').
\end{definition}
Notice that the suffix flip is defined (in particular) for any $(u,v)$ such that $\sum u_i \geq \sum v_i + r$, and the prefix flip is defined (in
particular) for any $u$ such that $\sum u_i \geq r/2$. Also, notice that $\sff{r}(u,v)$ is an involution (meaning that $\sff{r}\lbr \sff{r}(u,v)\rbr = (u,v)$), and hence, is injective. Moreover, as $\pf{r}$ is a composition of $\sff{r}$ (restricted to inputs of the form $(u,-u)$) with the map $x \mapsto (-x)$, it is injective as well.

\begin{remark}
We note that prefix/suffix flips are similar to the classical \emph{Andr\'{e}'s reflection method} (\cite{Andre1887}; see also~\cite{Renault08}), extensively used in enumerative combinatorics and other fields.
\end{remark}

\begin{definition}[Single coordinate flip]
	Let $x\in \spm^n$ be a vector whose `partial sums' $S_k(x) \ddd  \sum_{i=1}^{k} x_i$ satisfy $\max_k S_k(x)>0$. We define $\scf(x)\in \spm^n$ as follows.
    Take $t=\min\{i\in [n]\mid \forall j\colon S_i(x)\geq S_j(x)\}$ and set
	\[
	\scf(x)_{i}=\begin{cases}
	-x_{i}, & i=t\\
	x_{i}, & i \neq t
	\end{cases}.
	\]
	That is, we flip only a single coordinate of $x$ -- the first among the indices $k$ for which the partial sum $S_k(x)$ is maximal.
\end{definition}
The map $\scf$ is invertible, and thus, injective. To see this, note that if in the map $x \mapsto \scf(x)$, the $t$th coordinate was flipped, then the latest index in which the maximum of $\lbc S_k(\scf(x)) \rbc$ is attained is $t-1$. Indeed, by the definition of $\scf$, we have $S_{t-1}(\scf(x))=S_t(x)-1$, while for any $t'<t$ we have $S_{t'}(\scf(x)) \leq S_t(x)-1$ and for any $t'' \geq t$ we have $S_{t''}(\scf(x)) \leq S_t(x)-2$.

In particular, we can define an inverse mapping $\iscf$ as follows. Take $t=\max\{i\in\{0,\ldots,n-1\}\mid \forall j\colon S_i(x) \geq S_j(x)\}$ (where $S_0(x) \ddd 0$) and set
	\[
	\iscf(y)_{i}=\begin{cases}
	-y_{i}, & i= t+1\\
	y_{i}, & i \neq t+1
	\end{cases}.
	\]
We would like to use the map $\scf$ not only for $x$'s such that $\max_k S_k(\scf(x))>0$, but also for $x$'s for which we only know that $\sum a_i x_i >0$
for some non-negative weights $a_1,\ldots,a_n$. For this, we define the following variant which reorders the coordinates of $x$ according to the sizes of the $a_i$'s, applies $\scf$, and then reorders the coordinates back. 	
\begin{definition}\label{def:single_flip}
	For any $a\in\preals^{n}$ we define the partial map $\scf_a \colon \spm^n \to \spm^n$ as follows. Canonically choose a permutation $P_{a}\in S_n$ which satisfies $a_{P_a(1)}\geq\ldots\geq a_{P_a(n)}$ (e.g., using the lexicographic ordering on $S_{n}$). Then, define $\scf_a(x)=P_{a}^{-1}(\scf(P_a(x)))$.

\medskip
	
	Analogously, we define $\iscf_a \colon \spm^n\to\spm^n$ as $\iscf_a(y)=P_{a}^{-1}(\iscf(P_a(y)))$.
\end{definition}

Note that for any $a \in \mathbb{R}^n_{\geq 0}$, we have $\iscf_a\comp\scf_a=\mrm{id}_{\mrm{dom}(\scf_a)}$, and hence, $\scf_a$ is injective. We claim that the function $\scf_a(x)$ is defined
(in particular) for all $x$'s such that $\sum a_i x_i >0$. To see this, note that after the re-ordering of the coordinates of $x$, the function $\scf$ is applied on a vector $x'$ that satisfies $\sum a'_i x'_i>0$, where $a'_1\geq\ldots\geq a'_n$ (of course, the $a'_i$'s are a reordering of the $a_i$'s). The only reason $\scf(x')$ might not be defined, is if $\forall k:S_k(x') = \sum_{i=1}^{k} x'_i \leq 0$. But this cannot happen, from the following Abel's-summation argument:
	\[
	0 \underbrace{<}_{\text{assumption}} \sum_{i=1}^{n} a'_i x'_i = a'_n S_n(x') + \sum_{k=1}^{n-1} \lbr (a'_{k}-a'_{k+1}) S_k(x') \rbr \leq 0,
	\]
where the last inequality follows from the, apparently wrong, assumption $\forall k\colon S_k(x')\leq 0$.

\subsection{Proof of Lemmas~\ref{lem:intro-log_concavity} and~\ref{lem:decay_of_intervals}}
\label{sec:sub:lemmas-proof}

Now we are ready to present the proofs of the lemmas.

\medskip \noindent \textbf{Lemma~\ref{lem:intro-log_concavity}.} Let $X= \sum_i x_i$ be a sum of independent real random variables, and denote $F(t) = \pr[X>t]$. Set $m = \max_{i\in[n]}\{\sup {x_i}-\inf {x_i}\}$. For any $b\leq c\leq d$, we have
	\[
	F(d)F(b)\leq F(c)F(b+d-c-m).
	\]

\begin{proof}
	After subtracting $F(c)F(d)$ from both sides of Inequality~\eqref{eq:log_concavity_statement}, it is left to prove that
	\begin{equation}\label{Eq:Aux-Con1}
	\pr[X>d]\pr[X\in(b,c]]\leq\pr[X>c]\pr[X\in(b+d-c-m,d]].
	\end{equation}
	Let $\Omega$ be the underlying probability space over which $\{x_i\}_{i}$ and $X$ are defined. Without loss of generality, $\omega \mapsto \lbr x_i(\omega)\rbr_{i=1}^{n}$ is an injective map. We define an \emph{injective measure-preserving} map which takes as input
	a pair $\lbr \omega_{1},\omega_{2}\rbr \in\Omega^{2}$ for which
	$X\lbr \omega_{1}\rbr>d$ and $X\lbr \omega_{2}\rbr \in(b,c]$
	and outputs $\lbr \delta_{1},\delta_{2}\rbr \in\Omega^{2}$ for
	which $X\lbr \delta_{1}\rbr>c$ and $X\lbr \delta_{2}\rbr \in(b+d-c-m,d]$. This will  clearly conclude the proof.

Set $r=(d-c)-m$ (note that if $r \leq 0$ then the assertion holds trivially), and consider
	\[
	(v,u) = \sff{r}\lbr \lbr x_i(\omega_1) \rbr_{i=1}^{n}, \lbr x_i(\omega_2) \rbr_{i=1}^{n} \rbr.
	\]
Define $\delta_1, \delta_2$ by $\forall i:x_i(\delta_1)=u_i$ and $\forall i:x_i(\delta_2)=v_i$. We claim that $\psi\colon (\omega_1,\omega_2)\mapsto(\delta_1,\delta_2)$ is the desired map.
	
First, $\psi$ is well-defined: Since $X(\omega_1)>d$ and $X(\omega_2)\leq c$ we have $X(\omega_1)\geq X(\omega_2)+r$, so according to Definition~\ref{def:prefix_flip}, the suffix flip $\sff{r}$ is defined on the vectors $\lbr x_i(\omega_1) \rbr_{i=1}^{n}, \lbr x_i(\omega_2) \rbr_{i=1}^{n}$. Hence, $v,u$ are well-defined. As we assumed that $\omega \mapsto \lbr x_i(\omega)\rbr_{i=1}^{n}$ is injective, and since $v,u$ are obtained from $\lbr x_i(\omega_1) \rbr_{i=1}^{n}, \lbr x_i(\omega_2) \rbr_{i=1}^{n}$ by swapping some pairs of elements, $\delta_1,\delta_2$ are well-defined as well.

Second, $\psi$ is measure-preserving, as the variables $x_i$ are independent, and $\sff{r}$ just swaps pairs of identically distributed variables in its input. Furthermore, $\psi$ is injective, since $\sff{r}$ is invertible from the left (as noted after Definition~\ref{def:prefix_flip}).
	
Hence, it remains to show that the range of $\psi$ is included in the space represented by the r.h.s. of~\eqref{Eq:Aux-Con1}, i.e., that $X(\delta_1)>c$ and $X(\delta_2)\in(b+d-c-m,d]$. Observe that (unless $r<0$, in which case the assertion of the lemma is trivial), we have
	\[
	\sum_{i=1}^{n} x_i(\omega_1) - \sum_{i=1}^{n} u_i\in [r, r+m),
	\]
	according to how $\sff{r}$ is defined -- swapping the suffix just after the first index $t_0$ for which
\[
S_{t_0}(\lbr \lbr x_i(\omega_1) \rbr_{i=1}^{n}, \lbr x_i(\omega_2) \rbr_{i=1}^{n} \rbr) \geq r.
\]
(Here we also use the fact that each difference $x_i(\omega_1)-x_i(\omega_2)$ is bounded by $m$, and thus we have $S_{t_0}(\lbr \lbr x_i(\omega_1) \rbr_{i=1}^{n}, \lbr x_i(\omega_2) \rbr_{i=1}^{n} \rbr) \in [r,r+m)$.) Hence, $X(\delta_1)>X(\omega_1)-(r+m)>c$. Moreover,
	\[
	X(\delta_2)=X(\omega_1)+X(\omega_2)-X(\delta_1)\in \lbs r + X(\omega_2), r+m+X(\omega_2) \rbr \sub \lbr b+d-c-m, d \rbr,
	\]
	as required.
\end{proof}

\medskip \noindent \textbf{Lemma~\ref{lem:decay_of_intervals}.} Let $x_i$ be independent symmetric (around $0$) random variables with $\lba x_i\rba\leq a_i$, almost surely. For any $m\geq\max_i\{a_i\}$ and any $s,t$ such that $0\leq s\leq t$, we have
	\[
	\pr \lbs \sum_{i=1}^{n} x_i \in (t-m, t+m]\rbs \leq 5\pr \lbs \sum_{i=1}^{n} x_i \in (s-m, s+m]\rbs.
	\]

\begin{proof}
	Let $0\leq s \leq t$. We first show
\begin{equation}\label{Eq:Aux-Con1.5}
\pr_x \lbs \sum_{i=1}^{n} x_i \in (t-m, t+m] \rbs \leq \pr_x\lbs \sum_{i=1}^{n} x_i \in (s-m, s+3m] \rbs.
\end{equation}
We will do this with an injective measure-preserving map. Let $r=t-s-2m$ (note that the assertion holds trivially if $r \leq 0$). Denote $u_i=x_i$, and define $v = \pf{r}(u)$. By the definition of $\pf{r}$, for any $u$ we have
\begin{equation}\label{Eq:Aux-Con2}
\sum_{i=1}^{n} u_i - \sum_{i=1}^{n} \pf{r}(u)_i \in [r,r+ 2\max_i \{a_i\}) \subseteq [r,r+2m).
\end{equation}
Hence, in our case, if $\sum_{i=1}^{n} u_i\in (t-m,t+m]$ then $\sum_{i=1}^{n} v_i \in (s-m, s+3m]$. The map $\lbr x_i\rbr_{i=1}^{n} \to \lbr v_i\rbr_{i=1}^{n}$ is injective by Definition~\ref{def:prefix_flip}, and is measure-preserving as it only negates some $x_i$'s, which we assumed are \emph{symmetric} random variables. Therefore,~\eqref{Eq:Aux-Con1.5} holds, and thus, it is sufficient to show that
\begin{equation}\label{Eq:Aux-Con3}
\pr_x\lbs \sum_{i=1}^{n} x_i \in (s+m, s+3m]\rbs \leq 4\pr_x\lbs \sum_{i=1}^{n} x_i \in (s-m, s+m]\rbs.
\end{equation}
	\begin{sloppypar}
	For this, we construct four injective, measure-preserving maps from sub-events of $\lbc \sum_{i=1}^{n} x_i\in\lbr s+m,s+3m\rbs \rbc$ to the event
	\[
	\lbc \sum_{i=1}^{n} x_i\in\lbr s-m,s+m\rbs \rbc,
	\]
	whose union of domains covers $\lbc \omega \mid \sum_{i=1}^{n} x_i(\omega)\in\lbr s+m,s+3m\rbs \rbc $.
	\end{sloppypar}

\medskip
	
	Let $S=\lbc i\in[n]\mid \lba x_i\rba\leq\frac{m}{2}\rbc $ and $B=\lbc i\in[n]\mid \lba x_i\rba > \frac{m}{2}\rbc $. The maps we are going to construct will only negate some of the $x_i$'s, so $B,S$ are reconstructible from the output of any of the maps. Thus, it is sufficient to show that the maps are injective given the partition $B,S$. Moreover, these maps will be measure-preserving as they only negate input variables.
	\begin{enumerate}
		\item The first map $\psi_1$ is defined on inputs $\lbr x_i\rbr_{i=1}^{n}$ with $\sum_{i\in S} x_i \geq s+m$ and $\sum_{i\in[n]} x_i \in \lbr s+m, s+2m \rbs$. We set $r=m$, apply $\pf{r}\lbr \lbr x_i\rbr_{i\in S} \rbr$ on the $S$-coordinates, and leave the $B$-coordinates unchanged. The map is injective as $\pf{r}$ is. Also, the output $y=\psi_1(x)$ satisfies $\sum_{i\in[n]} y_i\in (s-m, s+m]$ by~\eqref{Eq:Aux-Con2} and the definition of $S$.
		\item The second map $\psi_2$ is defined on inputs $\lbr x_i\rbr_{i=1}^{n}$ with $\sum_{i\in S} x_i \geq s+m$ and $\sum_{i\in[n]} x_i \in \lbr s+2m, s+3m \rbs$. We set $r=2m$, apply $\pf{r}\lbr \lbr x_i\rbr_{i\in S} \rbr$ on the $S$-coordinates, and leave the $B$-coordinates unchanged. The map is injective and the output $y=\psi_2(x)$ satisfies $\sum_{i\in[n]} y_i\in (s-m, s+ m]$ exactly like in the previous case.
		\item The third map $\psi_3$ is defined on inputs $\lbr x_i\rbr_{i=1}^{n}$ with $\sum_{i\in S} x_i<s+m$. Notice that for such inputs, we have $\sum_{i\in B} x_i > 0$. First, for every $i\in B$, we extract $s_i=\sgn(x_i)$ and $b_i=\lba x_i\rba$. Then, we treat $\sigma \ddd \{s_i\}_{i \in B}$ as a $\{-1,1\}$-valued vector and $\{b_i\}_{i \in B}$ as a vector of weights. Since
\[
\sum_{i \in B} b_i s_i = \sum_{i \in B} |x_i| \sgn(x_i) = \sum_{i \in B} x_i > 0,
\]
we can apply the map $\scf_b$ to the vector $\sigma$ (as noted right after Definition~\ref{def:single_flip}), to obtain $\sigma'=\scf_b(\sigma)$. Finally, we define the output $y=\psi_3(x)$ by $\forall i\in B:y_i=\sigma'_i b_i$ and $\forall i\in S\colon y_i=x_i$. That is, we apply a single coordinate flip to the vector $(x_i)_{i \in B}$, and leave the $S$-coordinates unchanged. The map is injective since $\scf_b$ is. By the definitions of $\scf$ and of $B$, we have
\begin{equation}\label{Eq:Aux-Con4}
\sum_i x_i - \sum_i y_i \in (m,2m]
\end{equation}
(as $\scf$ flips a single coordinate whose value is between $m/2$ and $m$, being taken from $B$). Therefore, the output might not satisfy
$\sum_{i\in[n]} y_i \in (s-m,s+m]$. If it does not, we apply the next map $\psi_4$ on our intermediate ``output'' $y$.

\item The output of the previous map $y= \psi_{3}(x)$ has $\sum_{i\in[n]} y_i \in (s-m, s+2m)$. The fourth map $\psi_4$ is defined on inputs $\lbr x_i\rbr_{i=1}^{n}$ with $\sum_{i\in S} x_i<s+m$, for which $y=\psi_{3}(x)$ satisfies $\sum_{i\in[n]} y_i\not\in(s-m,s+m]$. In this case we set $\psi_4(x) \ddd \psi_3(y)$. Note that $y$ satisfies the conditions under which $\psi_3$ is defined. Indeed, we have $\sum_{i\in S} y_i < s+m$ since $\psi_3$ does not alter $S$-coordinates, and $\sum_{i\in[n]}y_i > s+m$, as by~\eqref{Eq:Aux-Con4} we have
\[
s-m < \sum_{i \in [n]} x_i - 2m \leq \sum_{i \in [n]} y_i,
\]
and we assumed that $\sum_{i\in[n]} y_i\not\in(s-m,s+m]$. Finally, we claim that $z=\psi_4(y)$ satisfies $\sum_{i\in[n]} z_i\in(s-m,s+m]$. Indeed, applying again~\eqref{Eq:Aux-Con4} we see that on the one hand,
\[
\sum_{i \in [n]} z_i \geq \sum_{i \in [n]} y_i - 2m > s-m,
\]
and on the other hand,
\[
\sum_{i \in [n]} z_i < \sum_{i \in [n]} y_i - m < \sum_{i \in [n]} x_i - 2m \leq s+m.
\]
\end{enumerate}
	This completes the proof of the lemma.
\end{proof}

Two remarks are due.

\medskip \noindent 1.~ One may wonder whether the constant $5$ in Inequality~\eqref{eq:decay_of_intervals} can generally be improved. We believe the correct value is $2$; it surely cannot be less. Indeed, consider $X=\sum_{i\in[n]} x_i$ where $x_i\sim\spm$ independently from each other. Trivially, $X$ assumes only values equal to $n\pmod{2}$. Hence, taking $n$ to be a large odd integer and $(m=1.5,s=1,t=2)$, we have
	\[
	\pr[X \in (t-m,t+m] = \pr[X \in (0.5, 3.5]] \approx 2\pr[X \in (-0.5, 2.5]] = 2\pr[X \in (s-m,s+m]].
	\]

\noindent 2.~ One may also wonder whether Inequality~\eqref{eq:decay_of_intervals} can be strengthened in the case where $m$ is large, so that the constant $5$ is replaced by $1+O(\max{a_i} / m)$. This can indeed be done, by slightly modifying the proof of Lemma~\ref{lem:decay_of_intervals}. Specifically, in the beginning of the proof we may define $r=t-s-2\max_i \{a_i\}$, and then due to~\eqref{Eq:Aux-Con2}, instead of~\eqref{Eq:Aux-Con1.5} we get
\begin{equation}\label{Eq:Aux-Con5}
\pr_x \lbs \sum_{i=1}^{n} x_i \in (t-m, t+m] \rbs \leq \pr_x\lbs \sum_{i=1}^{n} x_i \in (s-m, s+m+2\max_i \{a_i\}] \rbs.
\end{equation}
Hence, it is sufficient to prove that
\[
\pr_x \lbs \sum_{i=1}^{n} x_i \in (s+m, s+m+2\max_i \{a_i\}]\rbs \leq O \left(\frac{\max{a_i}}{m} \right) \pr_{x} \lbs \sum_{i=1}^{n} x_i \in (s-m, s+m]\rbs,
\]
and this indeed follows immediately by invoking Lemma~\ref{lem:decay_of_intervals} itself as a black-box.

\medskip

The following corollary of Lemma~\ref{lem:decay_of_intervals} will be used several times in the sequel, so for the sake of convenience we state it explicitly.
\begin{corollary}\label{cor:decay_of_influence}
	Let $f_{s}=\one\{a\cdot x>s\}$ be a family of halfspaces and suppose $a_1 \geq a_2 \geq \ldots a_n \geq 0$. Then for any $s,t \in \mathbb{R}$ such that $|s| \leq t$, we have $5\ii_1(f_s)\geq \ii_1(f_t)$.
\end{corollary}
\begin{proof}
	Notice that $\ii_1(f_r)=\pr_{x\sim\spm^n} [a\cdot x - a_1 x_1 \in (r-a_1, r+a_1]]$. Hence, for $0 \leq s <t$, the corollary follows immediately from Lemma~\ref{lem:decay_of_intervals}. To prove the assertion for $s<0<t$, note that if $s'=-s$, we have
\[
\ii_1(f_s) = \ii_1(\one\{a\cdot x \geq s'\}) = \pr_{x\sim\spm^n} [a\cdot x - a_1 x_1 \in [s'-a_1, s'+a_1)].
\]
Hence, it is sufficient to prove a variant of Lemma~\ref{lem:decay_of_intervals} in which the assertion is replaced by
\[
\pr \lbs \sum_{i=1}^{n} x_i \in (t-m, t+m]\rbs \leq 5\pr \lbs \sum_{i=1}^{n} x_i \in [s'-m, s'+m)\rbs
\]
(i.e., only the types of open-closed intervals are changed).
This variant can be proved by essentially repeating the proof of Lemma~\ref{lem:decay_of_intervals}; the only change required is slightly modifying the definition of the map $\sff{r}$, such that instead of using $t=\min \lbc i\in\lbc 0,\ldots,n\rbc \mid S_{i}\geq r \rbc$, one lets
$t=\min \lbc i\in\lbc 0,\ldots,n\rbc \mid S_{i}> r \rbc$. The rest of the argument (including the definition of the map $\scf_b$) works without change.
\end{proof}


\section{Local Chernoff Inequalities}
\label{sec:Chernoff}

In this section we prove our local Chernoff inequalities, namely, Theorems~\ref{thm:intro-strong_chernoff},~\ref{thm:intro-strong_chernoff2}, and~\ref{thm:intro-weak_chernoff}. First, for the sake of completeness we present the proof of Theorem~\ref{thm:intro-strong_chernoff} using the general method of Benjamini et al.~\cite{BKS03}, due to Devroye and Lugosi~\cite{DL08}, and then we present a proof of all three theorems, via Lemma~\ref{lem:intro-log_concavity}.

\subsection{Proof of Theorem~\ref{thm:intro-strong_chernoff}, using the Benjamini-Kalai-Schramm method}


Let us recall the formulation of Theorem~\ref{thm:intro-strong_chernoff}.

\medskip \noindent \textbf{Theorem~\ref{thm:intro-strong_chernoff}.} Let $\lbc x_i\rbc$ be independent random variables uniformly distributed in $\spm$, and let $a_i\in\preals$ be such that $\sum_i a_{i}^{2}=1$. There exists a universal constant $c>0$ such that if $t\geq 0$ and $\eps=\pr\lbs \sum_{i} a_i x_i > t\rbs$, then $\pr[\sum_{i} a_i x_i>t+\delta]\leq \frac{\eps}{2}$, for $\delta\leq \frac{c}{\sqrt{\log(1/\eps)}}$.

\mn To prove the theorem, one needs the following result of Talagrand~\cite{Tal94}, whose proof relies on the hypercontractive inequality~\cite{Bonami70}.
\begin{theorem}~\cite[Theorem~1.5]{Tal94}
\label{Thm:Tal}
Let $g$ be a real-valued function on the discrete cube, i.e., $g\colon\{-1,1\}^n \rightarrow \mathbb{R}$, and let $L_{i}g=(g(x)-g(x\xor e_{i}))/2$ for $i=1,2,\ldots,n$. Then
\begin{equation}\label{Eq:Tal-Aux}
	\var(g) \leq O\lbr \sum_{i} \frac{\normt{L_{i}g}}{1+\log(\normts{L_{i}g}/\lbn L_{i}g\rbn_{1})}\rbr.
\end{equation}
\end{theorem}

\begin{proof}[Proof of Theorem~\ref{thm:intro-strong_chernoff}, due to Devroye and Lugosi~\cite{DL08}]
We let $g=\max(t, a\cdot x)$ and $\eps=\pr[g>t]$, and apply Theorem~\ref{Thm:Tal} to $g$. Now, we bound the terms that appear in~\eqref{Eq:Tal-Aux}.

Firstly, we have
\[
\normt{L_{i}g} \leq \pr[L_{i}g \neq 0] \lbn L_{i}g \rbn_{\infty}^{2}.
\]
Clearly, $\forall i\colon L_i(g) \leq a_i$. Since $L_i g=0$ holds unless either $g(x)>t$ or $g(x \xor e_i)>t$, a union bound yields $\pr[L_{i}g \neq 0] \leq 2\eps$. Hence,
\[
\normt{L_{i}g} \leq \pr[L_{i}g \neq 0] \lbn L_{i}g \rbn_{\infty}^{2} \leq 2\eps \cdot a_{i}^{2}.
\]
Secondly, by the Cauchy-Schwarz inequality,
\[
\lbn L_{i}g \rbn_{1} \leq \sqrt{\pr[L_{i}g \neq 0]} \cdot \sqrt{\be[L_{i}g^{2}]} \leq \sqrt{2\eps} \normts{L_{i}g},
\]
and thus,
\[
1+\log(\normts{L_{i}g}/\lbn L_{i}g\rbn_{1}) \geq \log(1/\sqrt{2\epsilon}) = \Omega(\log(1/\eps)).
\]
Substituting into~\eqref{Eq:Tal-Aux} and using the assumption $\sum_i a_i^2=1$, we obtain
\begin{equation}\label{Eq:Tal-Aux2}
\var(g) \leq O\lbr \sum_{i} \frac{\normt{L_{i}g}}{1+\log(\normts{L_{i}g}/\lbn L_{i}g\rbn_{1})}\rbr \leq O\lbr \sum_i \frac{2\eps \cdot a_i^2}{\log(1/\eps)} \rbr \leq O(\eps/\log(1/\eps)).
\end{equation}
On the other hand, by Chebyshev's inequality, we have
\[
\pr[|g-\be[g]| > \sqrt{2\var(g)/\eps}] < \eps/2.
\]
Since for $t \geq 0$ we have $\pr[g = t] \geq 1/2 \geq \eps/2$, we must have $\be[g]\leq t + \sqrt{2\var(g)/\eps}$, and so we deduce
\[
\pr[g>t+\sqrt{8\var(g)/\eps}]\leq \eps / 2.
\]
Substituting the bound of~\eqref{Eq:Tal-Aux2} on $\var(g)$, we obtain $\pr[g>t+\delta]\leq \eps / 2$ for $\delta = O(1/\sqrt{\log(1/\eps)})$, as asserted.
\end{proof}

\subsection{Proof of Theorems~\ref{thm:intro-strong_chernoff},~\ref{thm:intro-strong_chernoff2}, and~\ref{thm:intro-weak_chernoff}, using Lemma~\ref{lem:intro-log_concavity}}

In the proof, we shall use Lemma~\ref{lem:intro-log_concavity} via the following auxiliary lemma.
\begin{lemma}\label{lem:log_concave_exp}
	Let $X=\sum x_{i}$ where $\{x_i\}$ are independent symmetric (around $0$) random variables with $\lba x_i\rba \leq a_i$ almost surely, and let $F(t)=\pr\lbs X>t\rbs$. Set $m=2\max_{i}\lbc a_{i}\rbc $. Then, for every $t\geq 0$ and $\delta > 0$, we have $F(t+\delta+m)^{l}\leq 2F(t)^{l+1}$, where $l=1+\left\lfloor t/\delta\right\rfloor$.
\end{lemma}
\begin{proof}
	Set $r = t-l\delta < 0$. Since $x_i$ are symmetric, we have $F(r)\geq \frac{1}{2}$. Applying repeatedly Lemma~\ref{lem:intro-log_concavity}, with $d=t+\delta+m$, $c=t$, and $b$ taken from the sequence of values $b=r, r+\delta, r+2\delta,\ldots, r+(l-1)\delta$, we get a series of inequalities:
	\bem
		\frac{1}{2}F(t+\delta+m)^{l}
		& \leq & F(r)F(t+\delta+m)^{l}\\
		& \underbrace{\leq}_{b=r} & F(r+\delta)F(t)F(t+\delta+m)^{l-1}\\
		& \underbrace{\leq}_{b=r+\delta} & F(r+2\delta)F(t)^{2}F(t+\delta+m)^{l-2}\\
		& \leq & \ldots\\
		& \underbrace{\leq}_{b=r+(l-1)\delta} & F(r+l\delta)F(t)^{l} = F(t)^{l+1}.
	\enm
This completes the proof.
\end{proof}

Now we are ready to present the proofs of the theorems. We note that although we already presented a proof of Theorem~\ref{thm:intro-strong_chernoff} above, we present an alternative proof as well, since it is more constructive and may be applicable in settings where Talagrand's result does not apply.

We begin with a proof of Theorem~\ref{thm:intro-weak_chernoff}.

\mn \textbf{Theorem~\ref{thm:intro-weak_chernoff}.} Let $X=\sum x_{i}$ where $\{x_i\}$ are independent symmetric (around $0$) random variables with $\lba x_i\rba \leq a_i$ almost surely, and let $F(t)=\pr\lbs X>t\rbs$. Set $m=2\max_{i}\lbc a_{i}\rbc $, and let $c\in(0,1)$. If $\epsilon=F(t)$ for some $t\geq0$, and $\delta\geq0$ is minimal such that $F(t+\delta)\leq c\cdot F(t)$, then we have
	\[
	\delta\leq m+O\lbr \log\lbr 2/c\rbr \sqrt{\frac{\sum_{i}a_{i}^{2}}{\log(1/\epsilon)}}\rbr.
    \]
\begin{proof}
Without loss of generality, assume $\sum a_i^2=1$. Denote $\eps=F(t)$. We split into two cases.

\medskip \noindent \textbf{Case 1: $t<\delta-m$}.
Let $\delta'\in(t+m,\delta)$. As $\delta'' \ddd \delta'-m>t$, we can apply Lemma~\ref{lem:log_concave_exp} with $\delta''$ and $l=1+\lfloor t/\delta'' \rfloor =1$ to obtain $F(t+\delta') = F(t+\delta''+m)\leq 2F(t)^{2}$.
By minimality of $\delta$ with respect to $F(t+\delta)\leq cF(t)$, we must have
\begin{equation}\label{Eq:Aux-Chernoff1}
2F(t)>c,
\end{equation}
as otherwise $F(t+\delta')\leq (2F(t)) \cdot F(t)\leq cF(t)$ with $\delta'<\delta$. Using again the minimality of $\delta$, we deduce that every $d<t+\delta$ satisfies $F(d)>cF(t)>c^{2}/2$. On the other hand, Hoeffding's inequality implies $F(d)\leq\exp\lbr -d^{2}/2\rbr $. Hence, $d<\sqrt{2\log\lbr 2/c^{2}\rbr }$. Since this holds for all $d<t+\delta$, we have $\delta\leq t+\delta\leq\sqrt{2\log\lbr 2/c^{2}\rbr }$.

Recall that by~\eqref{Eq:Aux-Chernoff1}, $\eps=F(t)>c/2$, and so, $\log(1/\eps)\leq\log(2/c)$. All in all,	we get the required inequality,
		\[
		\delta\leq\sqrt{2\log\lbr 2/c^{2}\rbr }<2\frac{\log\lbr 2/c\rbr }{\sqrt{\log\lbr 2/c\rbr }}\leq2\frac{\log\lbr 2/c\rbr }{\sqrt{\log\lbr 1/\eps\rbr }}.
		\]

\medskip \noindent \textbf{Case 2: $t\geq\delta-m$.}
If $\delta \leq m$, we are done. Otherwise, let $\delta' \in (0, \delta-m)$. Applying Lemma~\ref{lem:log_concave_exp} with $\delta'$ and $l=1+\lfloor t/\delta' \rfloor$, we get
\begin{equation}\label{Eq:Aux-Chernoff2}
F(t+\delta'+m)^{l}\leq 2F(t)^{l+1}.
\end{equation}
As $t\geq\delta'$, we have $l=1+\left\lfloor t/\delta'\right\rfloor \leq2t/\delta'$. Thus,~\eqref{Eq:Aux-Chernoff2} implies:
\begin{equation}\label{Eq:Aux-Chernoff3}
F(t+\delta'+m)\leq2^{1/l}F(t)^{1+1/l}\leq2F(t)\cdot F(t)^{\delta'/2t}.
\end{equation}
Since $F(t)=\eps$ by assumption, Hoeffding's inequality yields $t\leq\sqrt{2\log(1/\eps)}$. Substituting into~\eqref{Eq:Aux-Chernoff3}, we get
		\[
		F(t+\delta'+m)\leq 2F(t)\cdot F(t)^{\delta'/2t} \leq 2\eps\cdot\eps^{\delta'/\sqrt{8\log(1/\eps)}}.
		\]
On the other hand, by minimality of $\delta$ and since $\delta'+m<\delta$, we have
		$F(t+\delta'+m)>c\cdot\eps$, and so
\[
c/2< \eps^{\delta'/\sqrt{8\log(1/\eps)}} = \exp\lbr \log(\eps)\delta'/\sqrt{8\log(1/\eps)}\rbr.
\]
		Therefore, $\log\lbr 2/c\rbr >\delta'\sqrt{\log(1/\eps)/8}$, or more
		nicely, $\delta'<\sqrt{8}\frac{\log(2/c)}{\sqrt{\log(1/\eps)}}$. Since
		this inequality holds for every $\delta'\in(0,\delta-m)$, we have
		\[
		\delta\leq m+\sqrt{8}\frac{\log(2/c)}{\sqrt{\log(1/\eps)}}.
		\]
This completes the proof.
\end{proof}

We now prove Theorems~\ref{thm:intro-strong_chernoff} and~\ref{thm:intro-strong_chernoff2} together.

\mn \textbf{Theorems~\ref{thm:intro-strong_chernoff} and~\ref{thm:intro-strong_chernoff2}.} Let $\lbc x_i\rbc$ be independent random variables uniformly distributed in $\spm$, and let $a_i\in\preals$ be such that $\sum_i a_{i}^{2}=1$. There exists a universal constant $c>0$ such that if $t\geq 0$ and $\eps=\pr\lbs \sum_{i} a_i x_i > t\rbs$, then $\pr[\sum_{i} a_i x_i>t+\delta]\leq \frac{\eps}{2}$, for $\delta\leq \frac{c}{\sqrt{\log(1/\eps)}}$.
	
Furthermore, if $B,S$ is any partition of $[n]$ (which corresponds to `big' and `small' values of the $a_i$'s), then either $\lba B\rba\geq \frac{1}{2}\log(1/\eps)$, or  $\delta\leq c\sqrt{\frac{\sum_{i\in S} a_i^2}{\log(1/\eps)}}$.

\medskip

\begin{proof}
Let $X=\sum_{i\in[n]} a_i x_i$ and $\eps=\pr[X>t]$ for $t\geq 0$. Take $\delta\geq 0$ to be minimal such that $\pr_x[X > t+\delta] \leq \frac{\eps}{2}$.

\medskip \noindent First, we prove $\delta \leq O \lbr 1/\sqrt{\log(1/\eps)} \rbr$. Let
		\[
		\beta = 2/\sqrt{\log(1/\eps}), \qquad S=\{i\in[n]\mid a_i\leq \beta\}, \qquad \mbox{ and } \qquad B=[n]\sm S.
		\]
		We have $\lba B\rba \beta^2 \leq \sum_{i\in[n]} a_i^2 \leq 1$, and so $|B|\leq \frac{1}{4}\log(1/\eps)$. By the law of total probability,
\begin{equation}\label{Eq:Aux-Chernoff3.5}
\Pr_{x} \lbs a\cdot x>t \rbs =2^{-\lba B\rba}\sum_{y\in\spm^{B}}\Pr_{z\sim\spm^{S}} \lbs \sum_{i\in S}a_{i}z_{i}>t-\sum_{i\in B}a_{i}y_{i}\rbs.
\end{equation}
		Notice that $2^{-\lba B\rba}\geq \eps^{1/4}$, and so each of the probabilities on the right hand side is $\leq \eps^{3/4}$. For $y\in\spm^B$, let $\delta_y$ be minimal such that
		\[
		\Pr_{z} \lbs \sum_{i\in S}a_{i}z_{i}>t+\delta_y-\sum_{i\in B}a_{i}y_{i}\rbs \leq \frac{1}{2}\Pr_{z} \lbs \sum_{i\in S}a_{i}z_{i}>t-\sum_{i\in B}a_{i}y_{i}\rbs.
		\]
		Clearly, $\delta \leq \max\{\delta_y\mid y\in\spm^B\}$. Applying Theorem~\ref{thm:intro-weak_chernoff} to the random variable $Z=\sum_{i \in S} a_i z_i$, with $t-\sum_{i\in B}a_{i}y_{i}$ in place of $t$ and $\Pr_{z\sim\spm^{S}} \lbs \sum_{i\in S}a_{i}z_{i}>t-\sum_{i\in B}a_{i}y_{i}\rbs \leq \eps^{3/4}$ in place of $\eps$, we deduce
\begin{equation}\label{Eq:Aux-Chernoff4}
  \forall y\colon \delta_y \leq O \lbr 2\beta + \sqrt{\frac{\sum_{i\in S} a_i^2}{\log(1/\eps^{3/4})}} \rbr = O \lbr \frac{1}{\sqrt{\log(1/\eps)}} \rbr,
\end{equation}
as asserted.

\medskip \noindent Now, given a partition $B,S$ of $[n]$, we show that either $|B|\geq \frac{1}{2}\log(1/\eps)$, or $\delta \leq \sqrt{\frac{\sum_{i\in S} a_i^2}{\log(1/\eps)}}$. Assume that $|B|<\frac{1}{2}\log(1/\eps)$, and note that by~\eqref{Eq:Aux-Chernoff3.5}, we have
\begin{equation}\label{Eq:Aux-Chernoff5}
\forall y\colon \Pr_{z\sim\spm^{S}} \lbs \sum_{i\in S}a_{i}z_{i}>t-\sum_{i\in B}a_{i}y_{i}\rbs \leq \frac{\epsilon}{2^{-|B|}} \leq \eps^{1/2}.
\end{equation}
Let $\delta_y$ be as in the above proof, and consider the random variable $Z' = \alpha \sum_{i \in S} a_i z_i$, for
$\alpha=(\sum_{i \in S} a_i^2)^{-1/2}$ (which is needed for rescaling the weights to have sum-of-squares equal 1). Applying~\eqref{Eq:Aux-Chernoff4} to $Z'$, with $t'=\alpha(t-\sum_{i\in B}a_{i}y_{i})$ in place of $t$ and using $\Pr[Z'<t'] \leq \eps^{1/2}$ which follows from~\eqref{Eq:Aux-Chernoff5}, we get
\[
	\forall y:\delta_y \leq O \lbr \sqrt{\frac{\sum_{i\in S} a_i^2}{\log(1/\eps^{1/2})}} \rbr,
\]
and we conclude with $\delta \leq O \lbr \sqrt{\frac{\sum_{i\in S} a_i^2}{\log(1/\eps)}} \rbr$ since $\delta \leq \max_y \delta_y$, as above.
\end{proof}


\section{First-Degree Fourier Weight of Halfspaces}
\label{sec:First-Degree}

In this section we prove Theorem~\ref{Thm:Main-W^1} stating that any halfspace $f_{t}= \one\{a \cdot x > t\}$ satisfies $\wo(f_{t}) = \Omega(\mu(f_{t})^2 \log(1/\mu(f_{t})))$ (which is the maximal possible value up to a constant factor, by the aforementioned \emph{level-1 inequality}).

\medskip

We start with an easy lemma describing how `large coordinates' (i.e., coordinates $i$ for which $a_i$ is `large') influence a halfspace.
\begin{lemma}\label{lem:e_i_big_for_big_coordinates}
	Let $f_{t}(x)=\one\{a\cdot x > t\}$, where $a\in \preals^n$ and $t\geq 0$. Denote $\eps \ddd \mu(f_{t})=\pr_{x\sim\spm^n}[a\cdot x > t]$, and let $\beta \geq 0$ be minimal such that $\pr[a\cdot x > t+\beta]\leq \eps/3$. If $a_i > \beta/2$ for some $i\in[n]$, then $\ii_i(f_{t})\geq \frac{2}{3} \eps$.
\end{lemma}
\begin{proof}
	Define $F(s)=\pr[a\cdot x > s]$ and $G(s)=\pr[a\cdot x - a_i x_i > s]$. (Note that $G$ does not depend on the value of the coordinate $x_i$.) We have
	\[
		\ii_i(f_{t})=\pr[a\cdot x - a_i x_i \in (t-a_i, t+a_i]]=G(t-a_i)-G(t+a_i).
	\]
	By the definition of $G$,
\begin{align*}
F(s)&= \pr[a\cdot x > s] = \frac{1}{2} \left(\pr[a\cdot x > s | x_i=1] + \pr[a\cdot x > s | x_i=-1] \right) \\
&= \frac12 \left(G(s-a_i)+G(s+a_i) \right),
\end{align*}
and therefore, $G(s)=2F(s+a_i)-G(s+2a_i)$. Hence, $G(s)\leq 2F(s+a_i)$, and on the other hand, $G(s)\geq 2F(s+a_i) - 2F(s+3a_i)$.
Thus,
\[
	\ii_i(f_{t}) =
	G(t-a_i)-G(t+a_i) \geq
	2F(t)-4F(t+2a_i)\geq
	\frac{2}{3} \eps,
\]
where the last inequality holds since $a_i>\beta/2$.
\end{proof}

We proceed with a lemma which states that in some sense, the influence of a coordinate on a halfspace is `proportional' to its weight.
(Recall that there exist halfspaces for which the weight of some coordinate is positive and nevertheless, it has zero influence; the lemma shows that this
`anomaly' can be fixed by slightly modifying the function.)
\begin{lemma}\label{lem:e_i_lower_bound}
    For $\delta>0$ define
	$
	e_{i}^{\delta} \ddd \be_{s\sim U(0,\delta)} \lbs \ii_i(f_{t+s})\rbs,
	$
	where $U(0,\delta)$ is the uniform distribution over the interval $(0, \delta)$.
	We have
	\begin{equation}\label{eq:e_i_lower_bound}
		e_{i}^{\delta}\geq \frac {a_i}{\delta} \Pr \lbs a\cdot x\in [t+a_i,t+\delta-a_i]\rbs.
	\end{equation}
\end{lemma}
\begin{proof}
	By the definition of $f_{t+s}$, we have $\ii_{i}\lbr f_{t+s}\rbr =\Pr_{x}\lbs a\cdot x-a_{i}x_{i}\in(t+s-a_{i},t+s+a_{i}]\rbs $.
	Changing the order of integration, we can express $e_i^{\delta}$ in a different way:
	\bem
		e_{i}^\delta & = & \be_{s\sim U(0,\delta)}\be_{x}\lbs \one\lbc a\cdot x-a_{i}x_{i}\in(t+s-a_{i},t+s+a_{i})\rbc\rbs \\
		& = & \be_{x}\be_{s\sim U(0,\delta)}\lbs \one\lbc a\cdot x-a_{i}x_{i}\in(t+s-a_{i},t+s+a_{i})\rbc\rbs \\
		& = & \be_{x}\lbs \frac{1}{\delta}\int_{0}^{\delta}\one\lbc a\cdot x-a_{i}x_{i}\in(t+s-a_{i},t+s+a_{i})\rbc\mrm{ds}\rbs .
	\enm
	(Note that there is no difference here between closed-open segments and open segments, as there is no change in the integral involved.)
    We may assume $\delta\geq a_{i}$, for otherwise the statement of the lemma is trivial. An easy computation confirms that for any $t,r\in\reals$ we have
	\[
	\int_{0}^{\delta}\one\lbc r\in(t+s-a_{i},t+s+a_{i})\rbc\mrm{ds}\geq a_{i}\cdot \one\lbc r\in\lbs t,t+\delta\rbs \rbc.
	\]
	So we overall obtain
	\begin{equation}\label{eq:e_i_trivial_lower_bound}
		e_{i}^{\delta}\geq\frac{a_{i}}{\delta}\be_{x}\lbs \one\lbc a\cdot x-a_{i}x_{i}\in\lbs t,t+\delta\rbs \rbc\rbs .
	\end{equation}
	However, $\one\lbc a\cdot x-a_{i}x_{i}\in\lbs t,t+\delta\rbs \rbc\geq \one\lbc a\cdot x\in\lbs t+a_{i},t+\delta-a_{i}\rbs \rbc$,
	and so we deduce
	\[
	e_{i}^{\delta} \geq \frac{a_{i}}{\delta}\be_{x}\lbs \one \lbc a\cdot x\in\lbs t+a_{i},t+\delta-a_{i}\rbs \rbc\rbs \geq \frac{a_{i}}{\delta}\pr_x\lbs a\cdot x \in [t+a_i, t+\delta-a_i] \rbs.
	\]
This completes the proof.
\end{proof}

Now we are ready to prove that any halfspace has a `large' Fourier weight on the first degree.

\medskip \noindent \textbf{Theorem~\ref{Thm:Main-W^1}.} There exists a universal constant $c>0$ such that for any halfspace $f=\one(\sum_i a_i x_i>t)$ with $\be[f]\leq \frac{1}{2}$, we have
	\[
	c \mathbb{E}[f]^2 \log \frac{1}{\mathbb{E}[f]} \leq W^1(f) \leq 2 \mathbb{E}[f]^2 \log \frac{1}{\mathbb{E}[f]}.
	\]

\medskip

\begin{proof}
	Let $a\in\preals^{n}$ with $\sum a_i^2 = 1$, let $f_{t}(x)=\one\{a\cdot x > t\}$, and let $\eps=\be[f_{t}]$. We prove $\wo(f_{t}) \geq \Omega(\eps^2 \log(1/\eps))$.
	
The main idea of the proof is as follows. We divide the coordinates into a set $B$ of coordinates $i$ whose weight $a_i$ is `large', and a set $S$ of coordinates whose weight is `small' (the exact definition is given below). We show that either $|B|$ is `large', and then by Lemma~\ref{lem:e_i_big_for_big_coordinates}, the contribution of the coordinates in $B$ is already sufficient to guarantee $\wo(f_{t}) \geq \Omega(\eps^2 \log(1/\eps))$, or else, the contribution of the coordinates in $S$ will guarantee $\wo(f_{t}) \geq \Omega(\eps^2 \log(1/\eps))$. To show the latter (which is the more complex case), we note that by the Cauchy-Schwarz inequality, we have
	\begin{equation}\label{Eq:Aux-W1-1.5}
	\sqrt{\sum_{i\in S} a_i^2} \sqrt{\wo(f_{t})} \geq \sqrt{\sum_{i\in S} a_i^2} \sqrt{\sum_{i \in S} \widehat{f_{t}}(\{i\})^2} \geq \sum_{i\in S} a_i \ii_i(f_{t}).
	\end{equation}
Hence, it is sufficient to show that
	\begin{equation}\label{Eq:Aux-W1-2}
	\sum_{i\in S} a_i \ii_i(f_t)\geq \Omega\lbr \eps\sqrt{\log(1/\eps)}\cdot\sqrt{\sum_{i\in S}a_{i}^{2}}\rbr .
	\end{equation}
We will do so, using the local Chernoff inequality presented in Section~\ref{sec:Chernoff}. (Notice that the situation $\sum_{i \in S} a_i^2 = 0$ is contained in the `$|B|$ is large' case.)

\medskip

Let $\beta$ be minimal such that $F(t+\beta)\leq\eps/3$, and let $\gamma$ be minimal such that $F(t+\gamma)\leq\eps/6$. Denote $\delta=\beta+\gamma$, let $S=\left\{ i\in[n]\mid a_{i}\leq\beta\right\}$ and $B=[n]\sm S$.
	
    As in Lemma~\ref{lem:e_i_lower_bound},
    we denote $e_{i}^{\delta} \ddd \be_{s\sim U(0,\delta)} \lbs \ii_i(f_{t+s})\rbs$. For every $i\in S$, we apply Lemma~\ref{lem:e_i_lower_bound} to obtain
    \[
    	e_{i}^{\delta}\geq\frac{a_{i}}{\delta}\Pr\lbs a\cdot x\in\lbs t+a_{i},t+\delta-a_{i}\rbs \rbs.
    \]
	Since $a_{i}\leq\beta$, the definition of $\beta$ and $\gamma$ implies
	\begin{equation}\label{Eq:Aux-W1-1}
	e_{i}^{\delta}\geq\frac{a_{i}}{\delta}\Pr\lbs a\cdot x\in\lbs t+\beta,t+\gamma\rbs \rbs \geq\frac{a_{i}}{\delta}\cdot\frac{\eps}{6}.
	\end{equation}
	
	Applying the `strong local Chernoff inequality' (Theorem~\ref{thm:intro-strong_chernoff2}) to the function $f_{t}$, with $S,B$ as defined above, we obtain that either $\lba B\rba \geq\frac{1}{2}\lg\lbr 1/\eps\rbr $ or $\gamma\leq O\lbr \sqrt{\sum_{i\in S}a_{i}^{2}/\log(1/\eps)}\rbr $. (Formally, the theorem is applied three times, where $\Pr[a\cdot x>s]$ drops from $\eps$ to $\eps/2, \eps/4$, and then $\eps/6$ as $s$ increases.) In the latter case, we have $\delta \leq 2\gamma \leq O\lbr \sqrt{\sum_{i\in S}a_{i}^{2}/\log(1/\eps)}\rbr $.
	
	We consider three cases:
	
	\mn {\it Case 1:} $|B|$ is large -- specifically, $\lba B\rba \geq\frac{1}{2}\lg\lbr 1/\eps\rbr $. Note that by Lemma~\ref{lem:e_i_big_for_big_coordinates}, every $i\in B$ has $\ii_i(f_{t})\geq 2\eps/3$. Hence, in this case we have $\wo(f_{t})=\sum_{i} \ii_i(f_{t})^{2}\geq \frac{2}{9}\eps^2 \log(1/\eps)$, as asserted.
	
	\mn {\it Case 2:} $\eps\geq\frac{1}{4}$.

In this case, we use the aforementioned theorem of~\cite{GL94} which asserts that any halfspace $g_{r}:\{-1,1\}^n \rightarrow \{-1,1\}$ satisfies $W^{\leq 1}(g_{r}) \geq 1/2$. This yields $\wo(2f_{t}-1)+W^{0}(2f_{t}-1)\geq\frac{1}{2}$, which in turn implies
\[
	\wo(f_{t})\geq (1/2-(1-2\eps)^2)/4 \geq\frac{1}{16} \geq \frac{1}{3} \eps^2 \log(1/\eps),
\]
as asserted.
	
	\medskip \noindent {\it Case 3:} $\lba B\rba <\frac{1}{2}\lg\lbr 1/\eps\rbr $, $\delta\leq O\lbr \sqrt{\sum_{i\in S}a_{i}^{2}/\log(1/\eps)}\rbr $, and $\eps<\frac{1}{4}$. Using~\eqref{Eq:Aux-W1-1}, we have
	\begin{equation}\label{eq:large_w1_rough_inequality}
	\sum_{i\in S}a_{i}e_{i}^{\delta}\geq\frac{\eps}{6\delta}\sum_{i\in S}a_{i}^{2}.
	\end{equation}
	Since $\delta\leq O\lbr \sqrt{\nicefrac{\sum_{i\in S}a_{i}^{2}}{\log(1/\eps)}}\rbr $,
	we deduce
	\begin{equation}\label{Eq:Aux-W1-3}
	\sum_{i\in S}a_{i}e_{i}^{\delta}\geq\frac{\eps}{6\delta}\sum_{i\in S}a_{i}^{2}\geq\Omega\lbr \eps\sqrt{\log(1/\eps)}\cdot\sqrt{\sum_{i\in S}a_{i}^{2}}\rbr .
	\end{equation}
    Recall that by definition, $e_{i}^{\delta} = \be_{s\sim U(0,\delta)} \lbs \ii_i(f_{t+s})\rbs$, and thus, by linearity of expectation, we have
	\[
		\be_{s\sim U(0,\delta)} \lbs \sum_{i \in S} a_i \ii_i(f_{t+s})\rbs = \sum_{i \in S} a_i e_{i}^{\delta}.
	\]
    Hence,~\eqref{Eq:Aux-W1-3} implies that there exists $s\in(0,\delta)$ with
\begin{equation}\label{Eq:Aux-W1-4}
\sum_{i\in S} a_i \ii_i(f_{t+s})\geq \Omega\lbr \eps\sqrt{\log(1/\eps)}\cdot\sqrt{\sum_{i\in S}a_{i}^{2}}\rbr.
\end{equation}
To show that~\eqref{Eq:Aux-W1-2} holds, and thus complete the proof of the theorem, it is sufficient to show that the inequality~\eqref{Eq:Aux-W1-4} holds also for
$s=0$. This is achieved in the following proposition.
	\begin{proposition*}
		In the former settings, where $\lba B\rba \leq\frac{1}{2}\lg(1/\eps)$
		and $\eps<\frac{1}{4}$, for any $s>0$ we have $\sum_{i\in S}a_{i}\ii_{i}(f_{t})\geq\sum_{i\in S}a_{i}\ii_{i}\lbr f_{t+s}\rbr $.
	\end{proposition*}

\noindent {\it Proof of the Proposition.}~ We start by showing that $\sum_{i\in B}a_{i}\leq t$. Assume the
		contrary. We then have
		\[
		\eps = \Pr\lbs a\cdot x > t\rbs \geq\Pr\lbs \sum_{i\in B}a_{i}x_{i}>t\andd\sum_{i\in S}a_{i}x_{i}\geq0\rbs \geq\frac{1}{2}\cdot2^{-\lba B\rba }\geq\frac{1}{2}\sqrt{\eps}.
		\]			
		Hence, $\eps\ge1/4$, which contradicts the assumption.

\mn From the monotonicity of halfspaces, we have $\ii_{i}(f_{t+s})=\be_{x}\lbs x_{i} f_{t+s}(x) \rbs$, and so
\[
	\forall s:\sum_{i\in S}a_{i}\ii_{i}(f_{t+s})=\be_{x}\lbs \lbr \sum_{i\in S}a_{i}x_{i}\rbr \cdot f_{t+s}(x)\rbs .
\]
		Thus, as $\forall x\colon 0\leq f_{t+s}(x)\leq f_{t}(x)$, it is sufficient to show that $\sum_{i\in S}a_{i}x_{i}\geq0$
		whenever $f_{t}>0$. This indeed holds, as
\[
	(f_{t}(x)>0) \Rightarrow (a \cdot x>t) \Rightarrow \left(\sum_{i \in S} a_i x_i > t- \sum_{i \in B} a_i \geq 0 \right),
\]
where the ultimate inequality holds since $\sum_{i \in B} a_i \leq t$. This completes the proof of the proposition, and thus also the proof of the theorem.
\end{proof}


\section{The Maximal Influence of Halfspaces}\label{sec:lower_bound_I1}

In this section we prove Theorem~\ref{Thm:Main-influence}.

\medskip \noindent \textbf{Theorem~\ref{Thm:Main-influence}.} There exist universal constants $c_1,c_2$ such that for any halfspace $f= \one(\sum_i a_i x_i > t)$ with $\be[f]\leq \frac{1}{2}$ and $a_1 \geq a_2 \geq \ldots \geq a_n \geq 0$, we have
	\[
	c_1 \mathbb{E}[f] \min \{1,a_1 \sqrt{\log(1/\mathbb{E}[f])}\} \leq \max_{i} \ii_i(f) \leq c_2 \mathbb{E}[f] \min \{1,a_1 \sqrt{\log(1/\mathbb{E}[f])}\}.
	\]


\subsection{Proof of the Lower Bound}

We start with the lower bound, which follows directly from the tools developed in the previous sections.
\begin{proposition}\label{prop:large_i1-lower}
	Let $f_{t}= \one\{a \cdot x > t\}$ be a halfspace, where $\lbn a\rbn_2=1$ and $a_1 \geq a_2 \geq \ldots \geq a_n$, and assume $\mu(f_{t}) \leq 1/2$.
	We have
	\begin{equation}\label{Eq:Aux-Inf-1}
	\ii_{1}(f_{t}) \geq \Omega\lbr \mu(f_{t}) \min\lbr1, a_1 \sqrt{\log(1/\mu(f_{t}))}\rbr \rbr.
	\end{equation}
\end{proposition}

\begin{proof}
    Let $f_{t}$ satisfy the assumptions of the proposition. As in the proof of Theorem~\ref{Thm:Main-W^1},
    we let $\beta$ be minimal such that $F(t+\beta)\leq F(t)/3$ and $\gamma$ be minimal such that $F(t+\gamma)\leq F(t)/6$, and denote $\delta=\beta+\gamma$. We also
    denote $e_{i}^{\delta} \ddd \be_{s\sim U(0,\delta)} \lbs \ii_i(f_{t+s})\rbs$.

    If $a_1> \beta$, then by Lemma~\ref{lem:e_i_big_for_big_coordinates}, we have $\ii_1(f_{t})\geq \frac{2}{3}\mu(f_{t})$, as required in~\eqref{Eq:Aux-Inf-1}. Hence, we may assume $a_1\leq \beta$. By Lemma~\ref{lem:e_i_lower_bound}, we have
	\[
	e_1^{\delta} \underbrace{\geq}_{\text{Lemma}~\ref{lem:e_i_lower_bound}}\frac{a_1}{\delta}\Pr\lbs a\cdot x\in\lbs t +a_1,t+\delta-a_1\rbs \rbs \underbrace{\geq}_{a_1\leq\beta} \frac{a_1}{\delta}\Pr\lbs a\cdot x\in\lbs t +\beta,t+\gamma\rbs \rbs \underbrace{\geq}_{\text{Def. of } \beta,\gamma } \frac{a_1}{\delta} \frac{\mu(f_{t})}{6}.
	\]
	From Theorem~\ref{thm:intro-strong_chernoff} we get $\delta \leq O(1/\sqrt{\log(1/\mu(f_{t}))})$, and thus,
    \[
    e_1^{\delta} \geq \frac{a_1}{\delta} \frac{\mu(f_{t})}{6} \geq \Omega\lbr a_1\mu(f_{t})\sqrt{\log\lbr 1/\mu(f_{t})\rbr }\rbr.
    \]
    Since $e_1^{\delta}$ is defined as $\be_{s\sim U(0,\delta)}[\ii_1(f_{t+s})]$, this implies that there exists $s\geq 0$ with
    \[
    \ii_1(f_{t+s}) \geq \Omega\lbr a_1\mu(f_{t})\sqrt{\log\lbr 1/\mu(f_{t})\rbr }\rbr.
    \]
    Finally, by Corollary~\ref{cor:decay_of_influence}, for any $s>0$ we have $5\ii_1(f_{t})\geq\ii_1(f_{t+s})$. (Note that since $\mu(f_{t})\leq \frac{1}{2}$, we may assume $t\geq 0$ and so, Corollary~\ref{cor:decay_of_influence} can indeed be applied.) Hence,
    \[
    	\ii_1(f_{t}) \geq \frac15 \ii_1(f_{t+s}) \geq \Omega\lbr a_1\mu(f_{t})\sqrt{\log\lbr 1/\mu(f_{t})\rbr }\rbr,
	\]
    as asserted.
\end{proof}


\subsection{Proof of the Upper Bound}

To prove the upper bound, we use the following `reverse' version of Corollary~\ref{cor:decay_of_influence} which asserts that while $I_1(f_s)$ is a decreasing function of $s$ up to a constant factor, the `normalized' influence $I_1(f_s)/\be[f_s]$ is increasing up to a constant factor.
\begin{lemma}\label{lem:relative_inf_increase}
	Let $a\in\preals^{n}$ satisfy $a_1\geq \ldots \geq a_n$.
    For every $t\geq s\geq 0$ with $\be\lbs f_{t}\rbs >0$, we have
	\begin{equation}\label{Eq:Aux-Inf-2}
	\frac{\ii_{1}\lbr f_{s}\rbr }{\be\lbs f_{s}\rbs }\leq O\lbr \frac{\ii_{1}\lbr f_{t}\rbr }{\be\lbs f_{t}\rbs }\rbr .
	\end{equation}
\end{lemma}
\begin{proof}
    For any $r_1,r_2 \in \mathbb{R}$, let $G(r_1)=\Pr_{x\sim\spm^{n}}\lbs a\cdot x-a_{1}x_{1}>r_1\rbs $ and
    $G(r_1,r_2]=G(r_{1})-G(r_{2})$.
    Notice that for any $r$, we have
    \bem
		\be\lbs f_{r}\rbs & = & \frac{G(r-a_1)+G(r+a_1)}{2}\in G(r-a_1)\cdot\lbs \frac{1}{2},1\rbs,
		\\
		\ii_{1}\lbr f_{r}\rbr & = & G\lbr r-a_{1}\rbr -G\lbr r+a_{1}\rbr =G(r-a_1,r+a_1].
    \enm
    Hence, in order to prove~\eqref{Eq:Aux-Inf-2}, it is sufficient to show that for every $t\geq s\geq 0$ with $\be\lbs f_{t}\rbs >0$,
    \begin{equation}\label{Eq:Aux-Inf-3}
	\frac{G(s-a_1,s+a_1]}{G(s-a_1)}\leq O \left(\frac{G(t-a_1,t+a_1]}{G(t-a_1)} \right).
	\end{equation}
    Let $t\geq s\geq 0$. Lemma~\ref{lem:intro-log_concavity} (with $m=2a_1$)
	asserts that $G(b)G(d)\leq G(c)G\lbr b+d-c-2a_1\rbr $, for any $b\leq c\leq d$.
	Substitute $(b,c,d)=\lbr s-a_1,s+a_1,t+3a_1\rbr $ and subtract $G(c)G(d)$
	from both sides to deduce
	\[
	G(t+3a_1)G(s-a_1,s+a_1]\leq G(s+a_1)G(t-a_1,t+3a_1].
	\]
	It follows that
	\[
	\frac{G(t+3a_1)}{G(t-a_1,t+3a_1]}\leq\frac{G(s+a_1)}{G(s-a_1,s+a_1]}.
	\]
	Adding $1$ to both sides we obtain
	\[
	\frac{G(t-a_1)}{G(t-a_1,t+3a_1]}\leq\frac{G(s-a_1)}{G(s-a_1,s+a_1]}.
	\]
	Taking reciprocal and using Corollary~\ref{cor:decay_of_influence}, we get
	\[
	\frac{G(s-a_1,s+a_1]}{G(s-a_1)}\leq\frac{G(t-a_1,t+3a_1]}{G(t-a_1)}\underbrace{\leq}_{\text{Cor. }~\ref{cor:decay_of_influence}}\frac{6G(t-a_1,t+a_1]}{G(t-a_1)},
	\]
    and thus,~\eqref{Eq:Aux-Inf-3} holds, as asserted.
\end{proof}

We are now ready to prove the upper bound of Theorem~\ref{Thm:Main-influence}.

\begin{proposition}\label{prop:large_i1-upper}
	Let $f_{t}= \one\{a \cdot x > t\}$ be a halfspace, where $\lbn a\rbn_2=1$ and $a_1 \geq a_2 \geq \ldots \geq a_n$, and assume $\mu(f_{t}) \leq 1/2$.
	We have
	\begin{equation}\label{Eq:Aux-Inf-4}
	\ii_{1}(f_{t}) \leq O\lbr \mu(f_{t}) \min\lbr1, a_1 \sqrt{\log(1/\mu(f_{t}))}\rbr \rbr.
	\end{equation}
\end{proposition}

\begin{proof}
	Let $f_{t}$ satisfy the assumptions. Define $\beta,\gamma,$ and $\delta$ like in the proof of Proposition~\ref{prop:large_i1-lower}, and denote $\eps \ddd \mu(f_{t})$. Let $S=\left\{ i\in[n]\mid a_{i}\leq\beta\right\}$ and $B=[n]\sm S$. Note that we have
	\[
		\ii_1(f_{t})=2\be_{x}[x_1 f_{t}(x)]\leq 2\be[f_{t}(x)]\max_{x} |x_1| = 2\eps.
	\]
Therefore, if $a_{1}\sqrt{\log(1/\eps)}=\Omega(1)$ then we are done, since $\ii_{1}(f_{t})\leq O\lbr \eps\cdot\min \lbr 1,a_{1}\sqrt{\log(1/\eps)} \rbr \rbr $,	as claimed. Hence, for any required small universal constant $c_{1}>0$, we may assume
\begin{equation}\label{Eq:Aux-Inf-5}
a_{1} \leq \frac{c_{1}}{\sqrt{\log(1/\eps)}}.
\end{equation}

	By Lemma~\ref{lem:e_i_big_for_big_coordinates}, we have $\forall i\in B:\ii_{i}(f_{t})\geq\Omega(\eps)$, and in particular, $\wo(f_{t})\geq\Omega(\lba B\rba \eps^{2})$. Since we have $\wo(f_{t})\leq O(\eps^2 \log(1/\eps))$ by the aforementioned level-1 Inequality~\cite{Chang02,Talagrand96}, this implies $\lba B\rba \leq O(\log(1/\eps))$. Thus,
	\[
	\sum_{i\in B}a_{i}^{2}\leq\lba B\rba a_{1}^{2}\leq O(\log(1/\eps)a_{1}^{2}) \underbrace{\leq}_{\eqref{Eq:Aux-Inf-5}} \frac{1}{2},
	\]
	and hence,
    \begin{equation}\label{Eq:Aux-Inf-6}
    \sum_{i\in S}a_{i}^{2}\geq1/2.
    \end{equation}
    Following the proof of Theorem~\ref{Thm:Main-W^1}, and specifically Inequalities~\eqref{Eq:Aux-W1-1.5} and~\eqref{eq:large_w1_rough_inequality} (notice these inequalities hold regardless of the cases we had there), we deduce that there exists an $s\geq0$ with
	\[
	\wo(f_{t+s})\geq\Omega\lbr (\eps/\delta)^{2}\sum_{i\in S}a_{i}^{2}\rbr \underbrace{\geq}_{\eqref{Eq:Aux-Inf-6}} \Omega\lbr (\eps/\delta)^{2} \rbr.
	\]
	Using again the level-1 Inequality, we get $\delta=\Omega(1/\sqrt{\log(1/\eps)})$, as otherwise $\wo(f_{t+s})=\omega(\eps^2 \log (1/\eps))$, while $\be[f_{t+s}]\leq\eps$. In particular,
    \begin{equation}\label{Eq:Aux-Inf-6.5}
    \gamma\geq\delta/2\geq\Omega(1/\sqrt{\log(1/\eps)}).
    \end{equation}
    Hence, by~\eqref{Eq:Aux-Inf-5}, we may assume
    \begin{equation}\label{Eq:Aux-Inf-7}
    a_{1}<\gamma/4.
    \end{equation}
	Consider $e_1^{\gamma}=\be_{s\sim U(0,\gamma)}\lbs \ii_{1}(f_{t+s})\rbs $.
	Using Lemma~\ref{lem:relative_inf_increase}, we deduce
	\[
	e_1^{\gamma}=\be_{s\sim U(0,\gamma)}\lbs \frac{\ii_{1}\lbr f_{t+s}\rbr }{\be\lbs f_{t+s}\rbs }\be\lbs f_{t+s}\rbs \rbs \underbrace{\geq}_{\text{Lem. }\ref{lem:relative_inf_increase}}\Omega\lbr \frac{\ii_{1}(f_{t})}{\eps}\rbr \be_{s\sim U(0,\gamma)}\lbs \be\lbs f_{t+s}\rbs \rbs \underbrace{=}_{\gamma\text{ def.}}\Omega(\ii_{1}(f_{t})).
	\]
	On the other hand,
	\bem
		e_1^{\gamma}
		& = &
		\be_{s\sim U(0,\gamma)}\be_{x}\lbs \one\left\{ a\cdot x-a_{1}x_{1}\in(t+s-a_{1},t+s+a_{1}]\right\} \rbs
		\\
		& \underbrace{\leq}_{\text{Fubini}} &
		\frac{2a_{1}}{\gamma}\be_{x}\lbs \one\left\{ a\cdot x-a_{1}x_{1} > t-a_{1}\right\} \rbs
		\\
		& \leq &
		\frac{2a_{1}}{\gamma}\Pr_{x}\lbs a\cdot x>t-2 a_{1}\rbs
		\leq
		12\frac{\eps}{\gamma}a_{1},
	\enm
	where the last inequality will be justified below. Hence, $\ii_{1}(f_{t})\leq O\lbr \frac{\eps}{\gamma}a_{1} \rbr$, and thus, by~\eqref{Eq:Aux-Inf-6.5}, we have
    $\ii_{1}(f_{t})\leq O(a_1\eps\sqrt{\log(1/\eps)})$, as asserted.

    It only remains to justify why $\Pr_{x}\lbs a\cdot x>t-2a_{1}\rbs \leq6\eps$.
	This follows from Lemma~\ref{lem:intro-log_concavity}. Indeed,
    applying the lemma with $(b,c,d)=(t-2a_1,t,t+4a_1)$, we obtain
	\[
		F(t-2a_{1})F(t+4a_{1})\leq F(t)^{2}=\eps^{2}.
	\]
    Since $4a_1< \gamma$ by~\eqref{Eq:Aux-Inf-7}, the definition of $\gamma$ implies $F(t+4a_{1}) \geq \eps/6$. Hence,
    $\pr_x[a\cdot x > t-2a_1] = F(t-2a_{1})\leq6\eps$. This completes the proof of Proposition~\ref{prop:large_i1-upper}, and thus also the proof of Theorem~\ref{Thm:Main-influence}.
\end{proof}

\subsection{A Corollary of Theorem~\ref{Thm:Main-influence}}

We conclude this section with a corollary of Theorem~\ref{Thm:Main-influence} which essentially describes the probability that a linear form $l(x)=\sum a_i x_i$ (where $x_i\sim \spm$ uniformly and independently) lies in some interval $(a,b]$, by means of the tail probability $\pr[l(x) > a]$ and the interval length $|I|=b-a$. This corollary generalizes~\cite[Theorem 4]{Ser04}, up to the multiplicative constants.
We note that one could also prove this result directly, by an argument similar to that of the proof of Theorem~\ref{Thm:Main-influence}.
\begin{theorem}
	Let $l(x) = \sum a_i x_i$ be a linear form with $\lbn a\rbn_{2}=1$. If $m \geq \max_{i} |a_i|$, $t \geq -m$ and $\eps = \min(1/2,\pr_{x\sim\spm^n} \lbs l(x)>t \rbs)$, then
	\begin{equation}\label{eq:prob_linear_in_segment}
	\pr_{x\sim\spm^n} \lbs l(x) \in (t, t+2m] \rbs = \Theta \lbr \eps \min \lbr 1, m \sqrt{\log(1/\eps)} \rbr \rbr.
	\end{equation}
\end{theorem}
\begin{proof}
    Assume first $t \geq 0$. Consider the linear form $l'(x) = \frac{1}{\sqrt{1+m^2}} \left(m x_0 + l(x) \right)$ where $x_0\sim \spm$ is independent of the other variables. (Note that the normalization is intended to keep the sum-of-squares of the coefficients equal 1). Consider
    \[
    	g(x_0,x_1,\ldots,x_n)=\one \lbc l'(x)>\frac{1}{\sqrt{1+m^2}}(t+m) \rbc.
    \]
    We have $\ii_0(g)=\pr \lbs l(x) \in (t, t+2m] \rbs$, while
    \[
    	\mu(g) = (\pr[l(x)>t]+\pr[l(x)>t+2m]) / 2 \in [\eps/2, \eps].
    \]
	By Theorem~\ref{Thm:Main-influence}, we have
	\[
	\ii_{0}(g) = \Theta\lbr \mu(g) \min\lbr1, \frac{m}{\sqrt{1+m^2}} \sqrt{\log(1/\mu(g))}\rbr \rbr.
	\]
    Combining all these implies
	\begin{equation}\label{eq:prob_linear_in_segment_aux}
	\pr \lbs l(x) \in (t, t+2m] \rbs = \ii_{0}(g)= \Theta\lbr \eps \min\lbr1, \frac{m}{\sqrt{1+m^2}} \sqrt{\log(1/\eps)}\rbr \rbr.
	\end{equation}
	This concludes the proof, because $m=\Theta\lbr\nicefrac{m}{\sqrt{1+m^2}}\rbr$ unless $m \geq 1$, in which case the value of the minimum in Equation~\eqref{eq:prob_linear_in_segment_aux} is anyways $\Theta(1)$.

    It remains to consider the case $-m \leq t <0$. We claim that the assertion in this case follows from the assertion for $t=0$. On the one hand, let $l''(x)$ be a small enough perturbation of $l(x)$, such that $\pr[l''(x)=0]=0$ and the functions $\one \{l(x) > t\}$ and $\one \{l(x) > t+2m \}$ coincide with the functions $\one\{l''(x) > t\}$ and $\one \{l''(x) > t+2m \}$, respectively. (Such a perturbation exists, as explained in Section~\ref{sec:conventions}). Then,
    \[
    \pr[l(x)\in (t,t+2m]] \leq \pr[|l''(x)| < 2m] \leq 2\pr[l''(x)\in (0,2m)]\leq O \lbr \frac{1}{2}\min \lbr 1, m \sqrt{\log 2} \rbr \rbr,
    \]
    using the symmetry of $l''(x)$ and the case $t=0$ of~\eqref{eq:prob_linear_in_segment}.
    On the other hand, by Lemma~\ref{lem:decay_of_intervals} we have
    \[
    \pr \lbs l(x) \in (t, t+2m] \rbs \geq \frac{1}{5}\pr \lbs l(x) \in (0, 2m] \rbs \geq \Omega \lbr \frac{1}{2}\min \lbr 1, m \sqrt{\log 2} \rbr \rbr,
    \]
    where the ultimate inequality follows (again) from the case $t=0$. This completes the proof.
\end{proof}


\section{The Vertex Boundary of Halfspaces}\label{sec:boundary}

\begin{definition}
	Let $\lfunc{f}$ be a monotone Boolean function. For $\lambda\in\zo$, define the $\lambda$-vertex-boundary by
	\[
	\vb{f}{\lambda} = \pr_{x~\spm^n} \lbs (f(x)=\lambda) \andd (\exists i\in [n]\colon f(x\xor e_i)\neq \lambda) \rbs.
	\]
\end{definition}

In this section we prove Theorem~\ref{Thm:Main-boundary}.

\medskip \noindent \textbf{Theorem~\ref{Thm:Main-boundary}}. There exist universal constants $c'_1,c'_2$ such that for any halfspace $f= \one(\sum_i a_i x_i >t)$ with $\be[f]\leq \frac{1}{2}$ and $a_1 \geq a_2 \geq \ldots \geq a_n \geq 0$, we have
	\[
	c'_1 \mathbb{E}[f] \min \{1,a_1 \sqrt{\log(1/\mathbb{E}[f])}\} \leq |\partial(\one_f)|/2^n \leq c'_2 \mathbb{E}[f] \min \{1,a_1 \sqrt{\log(1/\mathbb{E}[f])}\}.
	\]
In addition, we show that for halfspaces, $\vb{f_{t}}{1}$ and $\vb{f_{t}}{0}$ cannot be too far from each other, while for general Boolean functions they can be `very' far.

\begin{remark}
All the results in this section apply also to halfspaces $f_{t}$ having $\mu(f_{t})\geq \frac{1}{2}$. The difference is that $\mu(f_{t})$ should be replaced by $1-\mu(f_{t})$ and $\vb{f_{t}}{0}$ exchanges roles with $\vb{f_{t}}{1}$.
\end{remark}

\subsection{Proof of Theorem~\ref{Thm:Main-boundary}}

We start with a proposition, which, together with Theorem~\ref{Thm:Main-influence}, implies Theorem~\ref{Thm:Main-boundary}.
\begin{proposition}\label{lem:upper_boundary}
	Let $f_{t}= \one\{a \cdot x > t\}$ be a halfspace with $a_1 \geq a_2 \geq \ldots \geq a_n > 0$. Assume that $\mathbb{E}[f_{t}] \leq \frac{1}{2}$. Then,
	\begin{equation}\label{eq:upper_boundary_ineq}
	\frac{1}{2}\ii_1(f_{t}) \leq \vb{f_{t}}{1}\leq \frac{7}{4} \ii_1(f_{t}).
	\end{equation}
\end{proposition}
\begin{proof}
Since $\mu(f_{t})\leq \frac{1}{2}$, we may assume w.l.o.g. $t>0$ as noted in Section~\ref{sec:conventions}. We observe that, since $a_{i}\geq a_{k}$ for all $k>i$, if for some $x$ we have $f_{t}(x)\neq f_{t}(x\xor e_k)$ and $i<k$ satisfies $x_i=x_k$, then $f_{t}(x)\neq f_{t}(x\xor e_i)$. Hence, setting
	\[
	c_{k}=\Pr_{x}\lbs \lbr x_{k}=1\rbr \andd \lbr f_{t}(x)\neq f_{t}(x\xor e_{k})\rbr \andd \Andd_{i<k}\lbr x_{i}=-1\rbr \rbs ,
	\]
	we have $\vb{f_{t}}{1}=\sum_{k=1}^{n}c_{k}$. Define auxiliary variables
	\[
	b_{k}\lbr \lambda_{1},\ldots,\lambda_{k-1}\rbr =\Pr_{x}\lbs \lbr f_{t}(x)\neq f_{t}(x\xor e_{k})\rbr \andd\bigwedge_{i<k}\lbr x_{i}=\lambda_{i}\rbr \rbs ,
	\]
	so that $c_{k}=\frac{1}{2}b_k(-1,\ldots,-1)$. Note that by the law of total probability,
    \[
    \ii_{k}(f_{t})=\sum_{\lambda\sim\spm^{k-1}}b_{k}\lbr \lambda\rbr.
    \]
    We claim that for any $\lambda\in\spm^{k-1}$, we have $b_{k}\lbr \lambda\rbr \ge\Omega\lbr c_{k}\rbr$. Indeed, we have
	\[
	b_k(\lambda)=2^{-(k-1)} \cdot \pr_{x\sim\spm^{[n]\sm[k]}} \lbs \sum_{i>k} a_i x_i \in \lbr t - \sum_{i < k} a_i \lambda_i \rbr + \lbr -a_k, a_k \rbs  \rbs.
	\]
    Hence, Corollary~\ref{cor:decay_of_influence} (applied to the family of halfspaces $\{\one\{\sum_{i=k+1}^n a_i x_i > s\}\}$, using the assumption $t>0$) implies:
    \[
    \forall \lambda \colon b_k(\lambda)\geq \frac{1}{5} b_k(-1,\ldots,-1)=\frac{2}{5} c_k.
    \]
	Since there are $2^{k-1}$ different $\lambda$'s, we have $\ii_{k}(f_{t}) = \sum_{\lambda\sim\spm^{k-1}}b_{k}\lbr \lambda\rbr \geq \frac{2^{k}}{5}c_k$, or more tightly, $\ii_{k}(f_{t})\geq \frac{2^{k}+8}{5}c_k$ (as the summand $b_k(-1,-1,\ldots,-1)$ contributes $2c_k$ instead of $\frac{2}{5} c_k$). Thus, $c_{k}\leq \frac{5}{2^k+8} \ii_{k}(f_{t})\leq \frac{5}{2^k+8}\ii_{1}(f_{t})$. Overall, we have
	\begin{equation}\label{Eq:Aux-Boundary1}
	\frac{1}{2}\ii_{1}(f_{t})=c_{1}\leq\vb{f_{t}}{1}=\sum_{k=1}^{n}c_{k} \leq
	\sum_{k=1}^{\infty} \frac{5}{2^k+8} \ii_1(f_{t})\leq 1.71 \ii_1(f_{t}),
	\end{equation}
	which completes the proof.
\end{proof}

\noindent Theorem~\ref{Thm:Main-boundary} follows immediately by combining Proposition~\ref{lem:upper_boundary} with Theorem~\ref{Thm:Main-influence}.

\begin{remark}
	The lower bound of~\eqref{eq:upper_boundary_ineq} is tight, e.g., for the dictatorship $\one(x_1>0)$. As is apparent from~\eqref{Eq:Aux-Boundary1}, the constant $7/4$ in the upper bound of~\eqref{eq:upper_boundary_ineq} is not tight. The majority function achieves $\vb{\maj_{n}}{1} \lessapprox 1 \cdot \ii_{1}(\maj_{n})$. Interestingly, there are examples of halfspaces $f_{t}$ with $\vb{f_{t}}{1} > \ii_{1}(f_{t})$. For instance, for an odd $n$, $f_{1}=\one\{5\cdot \sum_{i=1}^{4} x_i + 4\cdot\sum_{i=5}^{n}x_{i} > 1\}$ satisfies $\vb{f_{1}}{1}/\ii_{1}(f_{1}) \to 10/9$ when $n\to\infty$. Moreover, if Lemma~\ref{lem:decay_of_intervals} and Corollary~\ref{cor:decay_of_influence} are true with the conjectured constant of $2$ instead of $5$, then~\eqref{Eq:Aux-Boundary1} reads as
	\[
		\vb{f_{t}}{1} \leq \sum_{k=1}^{\infty} \frac{1}{2^{k-1}+1} \ii_{1}(f_{t}) \leq 1.27\ii_{1}(f_{t}).
	\]
	 We have no conjecture for what the correct upper bound for $\vb{f_{t}}{1}$ is.
\end{remark}

\subsection{A Relation Between Upper Boundary and Lower Boundary of Halfspaces}

\medskip We present now an argument similar to Proposition~\ref{lem:upper_boundary}, which establishes a sharp relation between $\vb{f}{0}$ and $\vb{f}{1}$ for halfspaces.
\begin{proposition}\label{thm:lower_boundary}
	 For any halfspace $f_{t}$ with $\mu(f_{t}) \leq \frac{1}{2}$, we have
	\begin{equation}\label{eq:lower_boundary}
	\Omega(\vb{f_{t}}{1}) \leq \vb{f_{t}}{0} \leq O(\log(1/\mu(f_{t}))) \vb{f_{t}}{1}
	\end{equation}
\end{proposition}

\begin{remark}
Note the left inequality in~\eqref{eq:lower_boundary} is tight for `Hamming balls'
\[
	f_t = \one \lbc \sum_{i=1}^n \frac{x_{i}}{\sqrt{n}} > t \rbc,
\]
with $n$ sufficiently large with respect to $t$. The right inequality in~\eqref{eq:lower_boundary} is tight for `subcubes'
\[
	f = \one \lbc \sum_{i=1}^{k} x_{i} > k-1/2 \rbc,
\]
as $\vb{f}{0}=k/2^k$, $\vb{f}{1}=1/2^k$ and $\mu(f)=1/2^k$.
\end{remark}

\begin{proof}[Proof of Proposition~\ref{thm:lower_boundary}]
Let $f_{t}= \one\{a \cdot x > t\}$ be a halfspace with $\pr[f_{t}=1] \leq \frac{1}{2}$, and assume without loss of generality $a_1 \geq a_2 \geq \ldots \geq a_n > 0$ and $t > 0$.

\medskip

We clearly have $\Omega(\vb{f_{t}}{1}) \leq \vb{f_{t}}{0}$, since $\vb{f_{t}}{0} \geq \frac{1}{2}\ii_{1}(f_{t})$ and $\ii_{1}(f_{t})=\Theta(\vb{f_{t}}{1})$ by Proposition~\ref{lem:upper_boundary}.

\medskip

For the proof of $\vb{f_{t}}{0} \leq O(\log(1/\mu(f_{t}))) \vb{f_{t}}{1}$, we let $b_k(\lambda):\spm^{k-1}\to [0,1]$ be the auxiliary variables from the proof of Proposition~\ref{lem:upper_boundary}, and set
	\[
		b'_k\lbr \lambda_{1},\ldots,\lambda_{k-1}\rbr \ddd \Pr_{x}\lbs \lbr f_{t}(x)=1\rbr \andd\bigwedge_{i<k}\lbr x_{i}=\lambda_{i}\rbr \rbs .
	\]
Recall that by the law of total probability we have $\ii_{k}(f_{t})=\sum_{\lambda \in\spm^{[k-1]}}b_{k}\lbr \lambda\rbr $, so in particular, $\forall \lambda \colon b_k(\lambda)\leq \ii_k(f_{t})\leq \ii_1(f_{t})$. In addition, like in the proof of Proposition~\ref{lem:upper_boundary}, we observe that $2\vb{f_{t}}{0}=	\sum_{k=1}^{n} b_k(1,\ldots,1)$.
\noindent We will soon prove
\begin{equation}\label{Eq:Aux-Boundary-2}
\forall k:b_k(1,\ldots,1)\leq O(2^{-k} \ii_1(f_{t}) / \mu(f_{t})).
\end{equation}
Combining~\eqref{Eq:Aux-Boundary-2} with the two former observations, we obtain
	\bem
	2\vb{f_{t}}{0}
	& = &
	\sum_{k=1}^{n} b_k(1,\ldots,1) \leq	\sum_{k=1}^{\infty} \min\lbc \ii_1(f_{t}), O(2^{-k} \ii_1(f_{t}) / \mu(f_{t})) \rbc
	\\
	& \leq &
	\sum_{k=1}^{\lg(1/\mu(f_{t}))} \ii_1(f_{t}) + \sum_{k'=1}^{\infty} O(\ii_1(f_{t})/2^{k'}) \leq O(\log(1/\mu(f_{t}))) \ii_1(f_{t})
	\\
	& \underbrace{\leq}_{\text{Prop.~\ref{lem:upper_boundary}}} &
	O(\log(1/\mu(f_{t}))) \vb{f_{t}}{1},	
	\enm
    completing the proof. Hence, it is only left to prove~\eqref{Eq:Aux-Boundary-2}. For this, consider the family of halfspaces
    \[
    \{f^{\lambda} \colon \{-1,1\}^{[n]\sm [k-1]} \rightarrow \{0,1\}\}_{\lambda \in \{-1,1\}^{[k-1]}},
    \]
    defined as $f^{\lambda}(x) \ddd f_{t}(\lambda,x)$. Note that for each $\lambda \in \{-1,1\}^{[k-1]}$, we have $b'_k(\lambda)=2^{1-k}\be[f^{\lambda}]$ and $b_k(\lambda)=2^{1-k}\ii_{\max}(f^{\lambda})$. Now, observe that
	\begin{equation}\label{eq:boundary_relative_inf}
		\forall \lambda \in \spm^{k-1}\colon \frac{b_k(1,\ldots,1)}{b'_k(1,\ldots,1)} \leq O \lbr \frac{b_k(\lambda)}{b'_k(\lambda)} \rbr.
	\end{equation}
    Indeed, an application of Lemma~\ref{lem:relative_inf_increase} to the family of halfspaces $\{f^{\lambda}\}$ implies that $\nicefrac{I_{\max}(f^{\lambda})}{\be[f^{\lambda}]}$ is (up to a constant) an increasing function of $t_{\lambda}\ddd t-\sum_{i <k}a_{i} \lambda_{i}$ in the range $t_{\lambda}\geq 0$; In the range $t_{\lambda}\leq 0$, $\be[f^{\lambda}]$ is clearly a decreasing function of $t_{\lambda}$, while from Corollary~\ref{cor:decay_of_influence} $I_{\max}(f^{\lambda})$ is (up to a constant) an increasing function of $t_{\lambda}$. This confirms Inequality~\eqref{eq:boundary_relative_inf}.
    Hence,
	\begin{align*}
	\frac{b_k(1,\ldots,1)}{b'_k(1,\ldots,1)}\mu(f_{t}) &\underbrace{=}_{\text{tot. prob.}} \frac{b_k(1,\ldots,1)}{b'_k(1,\ldots,1)} \sum_{\lambda\in \spm^{[k-1]}} b'_k(\lambda) \\
&\underbrace{\leq}_{\text{Eq.~\eqref{eq:boundary_relative_inf}}} \sum_{\lambda\in \spm^{[k-1]}} O \lbr b_k(\lambda) \rbr \leq O(\ii_k(f_{t})).
	\end{align*}
	Since $b'_k(1,\ldots,1)\leq 2^{1-k}$, we conclude $b_k(1,\ldots,1) \leq O(2^{-k} \ii_k(f_{t}) / \mu(f_{t}))$, as required by~\eqref{Eq:Aux-Boundary-2}. This completes the proof.
\end{proof}

\subsection{An Example Showing Discrepancy Between the Upper Boundary and the Lower Boundary, for General Boolean Functions}

We conclude this section with an example, suggested by Rani Hod, showing that for general Boolean functions, the difference between $\vb{f}{1}$ and $\vb{f}{0}$ can be very large (in contrast to Proposition~\ref{thm:lower_boundary}, which should be viewed as a property of halfspaces). The example is based on a random construction of Talagrand~\cite{Talagrand96}, originally proposed as an example of a monotone Boolean function $g$ with `maximal possible' vertex boundary $\mathrm{VB}(g)=\Omega(1)$ and `maximal possible' total influence $\ii(g)=\Omega(\sqrt{n})$.

\begin{example}\label{ex:vb_large_ratio}
	Let $b=\sqrt{n}$, $a = 2^{b}/b$, and define $\lfunc{h}$ by
	\[
		h(x) = \Orr_{i \in [a]} \Andd_{j \in S_i} x_{j},
	\]
	where, for every $i$, $S_{i}$ is a random subset of $[n]$ of size $b$. Also, let $f(x)=h(x)\orr \maj_n(x)$.
	We claim the following (proofs will be given below).
	\begin{enumerate}
		\item Always, $1/2 \leq \mu(f) \leq 1/2+1/\sqrt{n}$.
		\item Always, $\vb{f}{1} = O(1/\sqrt{n})$.
		\item With a probability that is bounded away from zero, $\vb{f}{0} = \Omega(1)$.
	\end{enumerate}
	Hence, there exists an almost unbiased Boolean function $f$ having a multiplicative gap of $\sqrt{n}$ between $\vb{f}{0}$ and $\vb{f}{1}$. (For comparison, Proposition~\ref{thm:lower_boundary} implies that for almost unbiased halfspaces we have $\vb{f}{0} = \Theta(\vb{f}{1})$.) Furthermore, the example implies that apparently, there is no analog to Proposition~\ref{lem:upper_boundary} for general functions, as $f$ and its dual function $1-f(-x)$ are similar in terms of Fourier expansion (and in particular, have the same influences), but are very different with respect to the $\vb{\cdot}{1}$ measure.

\medskip \noindent Let us verify the above claims.
	\begin{enumerate}
		\item It is clear that $\mu(f) \geq \mu(\maj_n)=1/2$. On the other hand, since $\mu(\Andd_{j \in S_i} x_{j}) = 2^{-b}$, from a union bound we have $\mu(h) \leq \frac{1}{b}$ and consequently,
\[
\mu(f)\leq \mu(\maj_n)+\mu(h) \leq \frac{1}{2} + \frac{1}{b} = \frac{1}{2}+\frac{1}{\sqrt{n}}.
\]
		\item For an $x \in \spm^{n}$ to be in the upper-boundary of $f$, either it is in the upper-boundary of $h$ or in that of $\maj_{n}$. We have $\vb{\maj_{n}}{1}=\Theta(1/\sqrt{n})$, and $\vb{h}{1}\leq \mu(h) \leq 1/b \leq 1/\sqrt{n}$ as above. Hence, $\vb{f}{1} = O(1/\sqrt{n})$.
		\item For an $x \in \spm^{n}$ to be in the lower-boundary of $f$, it is sufficient that $x$ is in the lower-boundary of $h$ and $\maj_{n}(x)=0$. Let $x \in \spm^{n}$ be chosen uniformly at random among the vectors that satisfy $\sum x_i = -2c\sqrt{n}$, for a fixed $c \in (0,10)$. We want to show that with some positive probability (`continuously') depending on $c$, $x$ lies in the lower-boundary of $h$. As $\mu(\{x \mid -20\sqrt{n} \leq \sum x_i <0\})=\Omega(1)$, this will imply $\vb{f}{0} = \Omega(1)$.

Since $\mu(h) =o(1)$ as we showed above, it is sufficient to show that
		\begin{equation}\label{Eq:Aux-Boundary3}
\pr_x \lbs \exists i \in [a]\colon \lba S_{i} \sm \mrm{supp}(x) \rba \leq 1\rbs = \Omega(1).
        \end{equation}
		Consider a specific $i \in [a]$. The probability $P_{i}(x)$ that $\lba S_{i} \sm \mrm{supp}(x) \rba \leq 1$ given $X = \lba \mrm{supp}(x)\rba$, is at least $(n-X)\binom{X}{b-1} / \binom{n}{b}$. Recall $X = n/2 - c\sqrt{n}$, so a computation gives $P_{i}(x) \gtrsim b 2^{-b-1} \exp(-cb / \sqrt{n})$ for $c=O(1)$. Since there are $a=2^{b}/b$ independent relevant $i$'s, the number of those $i$ with $\lba S_{i}\sm \mrm{supp}(x) \rba \leq 1$ is approximately distributed $\mrm{Poi}(\exp(-O(c)))$, and is nonzero with a constant probability. Thus,~\eqref{Eq:Aux-Boundary3} holds, as asserted.
	\end{enumerate}
\end{example}


\section{\texorpdfstring{$k$}{}th Degree Fourier Weight of Halfspaces}
\label{sec:W^k}

The classical \emph{level-$k$ inequality}~\cite[Section 9.5]{O'D14} asserts the following.
\begin{theorem}[Level-k inequality]\label{Thm:Level-k}
For any $k \in \mathbb{N}$ and for any halfspace $f_{t}=\one\{a\cdot x>t\}$ with $\eps=\mu(f_{t})<e^{-k/2}$, we have
	\begin{equation}\label{Level-k}
	\wk{\leq k}(f_{t})\leq \left(\frac{2e}{k}\right)^k \eps^{2}\log(1/\eps)^{k}.
	\end{equation}
\end{theorem}
In this section we prove Theorem~\ref{Thm:Main-W^k} which asserts that the level-$k$ inequality is tight for strongly biased halfspaces, up to a multiplicative factor depending only on $k$. Specifically, we prove the following result which clearly includes Theorem~\ref{Thm:Main-W^k}.


\begin{theorem}\label{thm:large_wk}
	There exist universal constants $c_1,c_2$ such that for any $k \in \mathbb{N}$ and for any halfspace $f_{t}=\one\{a\cdot x>t\}$ with $\eps=\mu(f_{t})<2^{-999k}$ and $\ii_{1}(f_{t})\leq c_1 \mu(f_{t})/k$, we have
	\begin{equation}\label{eq:wk_is_large}
	\wk{k}(f_{t})\geq \frac{(c_{2}\log(2k)) ^ {-k}}{k!} \eps^{2}\log(1/\eps)^{k}.
	\end{equation}
\end{theorem}

\begin{remark}
Let us compare Theorem~\ref{thm:large_wk} with Theorem~\ref{Thm:Main-W^1}. While Theorem~\ref{Thm:Main-W^1} expresses the tightness of the level-1 inequality for halfspaces (up to a multiplicative constant factor), Theorem~\ref{thm:large_wk} states that even the level-$k$ inequalities are essentially tight for halfspaces. On the other hand, Theorem~\ref{thm:large_wk} has three disadvantages. The first is the requirement that all the influences $\ii_{i}(f_{t})$ are somewhat small. (Note that the maximal possible value of an influence is $2\mu(f_{t})$, and so, we are `missing' a factor of $O(k)$.)
The second is that Equation~\eqref{eq:wk_is_large} is not the exact converse of the level-$k$ inequality, as there is a factor $\log(2k)^{-k}$ off. The third is that we require $\mu(f_t)$ to be at most $2^{-999k}$, instead of $e^{-k/2}$ in the level-$k$ inequality, that is, a factor $c^k$ off. We note however that for a constant $k$ (which is the case highlighted in Theorem~\ref{Thm:Main-W^k}),~\eqref{eq:wk_is_large} is indeed tight up to a constant multiplicative factor.

We believe the deficiencies of Theorem~\ref{thm:large_wk} are actually not inherent, and are side effects of our proof. Specifically, it is plausible one can omit the assumption that $f_{t}$'s influences are small, replace the multiplicative $(c_{2}\log(2k)) ^ {-k}$ in Equation~\eqref{eq:wk_is_large} by $c^{k}$ for some universal constant $c>0$, and claim that the assertion holds whenever $\mu(f_t) \leq 2^{-k}$.
\end{remark}

\medskip This section is organized as follows. In Section~\ref{sec:sub:perturbation} we present a study a certain `$k$-degree perturbation' of influences, that generalizes the perturbation we used in the proof of Theorem~\ref{Thm:Main-W^1}. In Section~\ref{sec:sub:proof-W^k} we use these $k$-degree perturbations to prove Theorem~\ref{thm:large_wk}, modulo two auxiliary claims. These claims are proved in Section~\ref{sec:sub:auxiliary}.

\subsection{A k-Degree Smoothing of Influences}
\label{sec:sub:perturbation}

Let $f_t=\one\{a\cdot x>t\}$ be a halfspace. Recall that in the proof of Theorem~\ref{Thm:Main-W^1}, a central role is played by the quantity $\sum a_i e_{i}^{\delta}$, where
\[
e_{i}^{\delta} \ddd \be_{s\sim U(0,\delta)} \lbs \ii_i(f_{t+s})\rbs = \be_{s\sim \delta \cdot U(0,1)} [\widehat{f_{t+s}}(\{i\})].
\]
In this subsection we consider the following degree-$k$ generalization of this notion.

For a set $S$, we denote $a^S = \prod_{i \in S} a_i$. We let $T=T_k$ be a random variable, distributed as the sum of $k$ independent $U(0,1)$-distributed variables (also called `the Irwin-Hall distribution'). Then, for any $S \subset [n]$ with $|S|=k$ we set
\begin{equation}\label{Eq:Aux-W^k-1}
e_{S}^{\delta} \ddd \be_{s\sim \delta \cdot T} [\widehat{f_{t+s}}(S)],
\end{equation}
and consider the quantity
$$
M = \sum_{|S|=k} a^S e_{S}^{\delta}.
$$
The following propositions will help us to study this quantity.

\begin{claim}\label{prop:irwin_hall_distribution}
	Let $X=\sum_{i=1}^{m} X_{i}$ be the sum of $m$ independent random variables distributed $X_i\sim U(0,1)$. Let $G_{m}=G_{X}$ be the cumulative distribution function of $X$.
	Then the $m$th derivative of $G_{m}$ satisfies $G_{m}^{(m)}(x)=(-1)^{\lfloor x \rfloor}\binom{m-1}{\lfloor x \rfloor}$, except for $x\in \{0,\ldots,m\}$.
\end{claim}
\begin{proof}
	We have $G_{m}(x)=\int_{x-1}^{x} G_{m-1}(t) dt$, and so $G_{m}^{(m)}(x)=G_{m-1}^{(m-1)}(x)-G_{m-1}^{(m-1)}(x-1)$. The assertion follows by induction.
\end{proof}

\begin{claim}\label{prop:fourier_is_derivative}
	Let $g\colon\reals\to\reals$ be a function that is differentiable $m$ times. Let $a \in \preals^{m}$ and denote $S=\sum_{i\in[m]} a_i$. For any $s\in \reals$, we have
	\begin{equation}\label{eq:fourier_is_derivative}
	\prod_{i=1}^{m} a_i \inf_{t\in (s-S, s+S)} g^{(m)}(t) \leq \be_{x\sim \spm^{m}} \lbs x^{[m]}g(s+a\cdot x) \rbs \leq \prod_{i=1}^{m} a_i \sup_{t\in (s-S, s+S)} g^{(m)}(t),
	\end{equation}
where $x^{[m]} = \prod_{i=1}^m x_i$.
\end{claim}
\begin{proof}
	Integrating each time the inner-most integral, one can find by induction that
	\[
	\sum_{x\in \spm^{m}} x^{[m]}g(s+a\cdot x) = \prod_{i=1}^m a_i \int_{-1}^{1} \cdots \int_{-1}^{1} g^{(m)}\lbr s+\sum_{i=1}^{m} a_i t_i \rbr dt_m \cdots dt_1,
	\]
	and so, Equation~\eqref{eq:fourier_is_derivative} follows.
\end{proof}

\begin{proposition}\label{prop:e_s_lower_bound}
	Let $f_t(x)=\one\{a\cdot x > t\}$ be a halfspace with $\forall i\colon a_i \geq 0$, let $S \subset [n]$ be a set of size $k$, let $\delta>0$ and let $e_{S}^{\delta}$ be as defined in~\eqref{Eq:Aux-W^k-1}. If $A\ddd \sum_{i\in S} a_i \leq \delta/2$, then
	\begin{equation}\label{eq:e_s_lower_bound}
	e_{S}^{\delta}\geq\frac{\prod_{i\in S}a_{i}}{\delta^{k}}\lbr \Pr_{z\in\spm^{[n]}}\lbs a\cdot z > t+2A \rbs- \sum_{l=1}^{k} \binom{k}{l}\Pr_{z\in\spm^{[n]}}\lbs a\cdot z>t+l\delta-2A\rbs \rbr.
	\end{equation}
\end{proposition}

\begin{proof}
	Notice we have, by definition,
	\begin{equation}\label{Eq:Aux-W^k-1.5}
	\widehat{f_{t+s}}(S)=\be_{y\in\spm^{[n]\backslash S}}\be_{x\in\spm^{S}}\lbs x^{S}\cdot\one\lbc a\cdot y+ a\cdot x-t>s\rbc\rbs,
	\end{equation}
	with the somewhat abusive notation $a\cdot x=\sum_{i\in S}a_{i}x_{i}$ and $a\cdot y=\sum_{i\notin S}a_{i}y_{i}$.
	Recall that $T$ is defined as the sum of $k$ independent $U(0,1)$-distributed variables, and let $G_{k}$ be the cumulative distribution function of $T$. We have
\[
\be_{T} \lbs \one\lbc a\cdot y+ a\cdot x-t> \delta T\rbc \rbs = \Pr_T \lbs T<\frac{a \cdot y+a \cdot x -t}{\delta} \rbs = G_{k} \left( \frac{a \cdot y+a \cdot x -t}{\delta} \right).
\]
Hence, substituting into~\eqref{Eq:Aux-W^k-1.5} and using Fubini's theorem, we obtain
	\begin{align}\label{Eq:Aux-W^k-2}
	e_{S}^{\delta}=	\be_{T}\lbs \widehat{f_{t+\delta \cdot T}}(S)\rbs=
	\be_{y\in\spm^{[n]\sm S}}\be_{x\in\spm^{S}}\lbs x^{S}G_{k}\lbr \frac{a\cdot x+a\cdot y-t}{\delta} \rbr \rbs.
    \end{align}
In the right hand side, for each fixed $y$, the expectation over $x$ has the form $\be_{x\sim \spm^{m}} [ x^{[m]}g(s+a\cdot x) ]$ discussed in Claim~\ref{prop:fourier_is_derivative}, with $(k,S,G_{k},(a\cdot y-t)/\delta,a/\delta)$ in place of $(m,[m],g,s,a)$, respectively. Thus, by Claim~\ref{prop:fourier_is_derivative}, each such expectation can be bounded from below by
\[
\prod_{i \in S} \lbr \frac{a_i}{\delta} \rbr \cdot \inf_{t\in ((a\cdot y-t-A)/\delta, (a\cdot y-t+A)/\delta)} G_{k}^{(k)}(t).
\]

By Claim~\ref{prop:irwin_hall_distribution}, the $k$th derivative $G_{k}^{(k)}$ is piece-wise constant, and specifically, satisfies $G_{k}^{(k)}(x)=(-1)^{\lfloor x \rfloor}\binom{k-1}{\lfloor x \rfloor}$. Hence, we can partition the $y$'s into subsets, such that inside each subset, $G_{k}^{(k)}$ is `almost' constant in the range $((a\cdot y-t-A)/\delta, (a\cdot y-t+A)/\delta)$.

The `main' subset we consider is $\{y \in \spm^{n} \mid a\cdot y-t\in\lbs A,\delta-A\rbs\}$, for which we have $((a\cdot y-t-A)/\delta, (a\cdot y-t+A)/\delta) \subset (0,1)$, and thus, $G_{k}^{(k)}(r)= {\binom{k-1}{0}} = 1$ for all $r \in ((a\cdot y-t-A)/\delta, (a\cdot y-t+A)/\delta)$. Therefore, for all $y$'s of this subset we have $\be_{x\in\spm^{S}}\lbs x^{S}G_{k}\lbr \frac{a\cdot x+a\cdot y-t}{\delta} \rbr \rbs \geq \frac{\prod_{i \in S} a_i}{\delta^k}$. The other subsets correspond to $y$'s for which $(a\cdot y-t+A)/\delta \in [j,j+1)$; as we are interested only in a lower bound, we may take the contributions of all these subsets with a `$-$' sign, and enlarge each such set of $y$'s for the sake of simplicity. Doing so and substituting into~\eqref{Eq:Aux-W^k-2}, we get
    \begin{align}\label{Eq:Aux-W^k-2.5}
    \begin{split}
    e_{S}^{\delta} & \geq 			
    \frac{\prod_{i\in S}a_{i}}{\delta^{k}}\cdot
	\Big( \Pr_{y\in\spm^{[n]\backslash S}}\lbs a\cdot y-t\in\lbs A,\delta-A\rbs \rbs -
	\\
	& \qquad \qquad \qquad
	- \sum_{l=1}^{k-1} \binom{k-1}{l}
	\Pr_{y\in\spm^{[n]\backslash S}}\lbs a\cdot y-t>l\delta-A\rbs \Big).
	\end{split}
    \end{align}	
One can easily obtain the following two (crude) inequalities:
	\begin{eqnarray*}
	\Pr_{y\in\spm^{[n]\sm S}}\lbs a\cdot y\in\lbs t+A,t+\delta-A\rbs \rbs
	& \geq &
	\Pr_{z\in\spm^{[n]}}\lbs a\cdot z\in\lbs t+2A,t+\delta-2A\rbs \rbs;
	\\	
	\Pr_{y\in\spm^{[n]\sm S}}\lbs a\cdot y>t+l\delta-A\rbs
	& \leq &
	\Pr_{z\in\spm^{[n]}}\lbs a\cdot z>t+l\delta-2A\rbs.
	\end{eqnarray*}
Substituting into~\eqref{Eq:Aux-W^k-2.5}, we obtain:
	\begin{align*}
	e_{S}^{\delta}&\geq\frac{\prod_{i\in S}a_{i}}{\delta^{k}} \Big( \Pr_{z\in\spm^{[n]}}\lbs a\cdot z\in\lbs t+2A,t+\delta-2A\rbs \rbs - \\
    & \qquad \qquad \qquad
	-\sum_{l=1}^{k-1} \binom{k-1}{l} \Pr_{z\in\spm^{[n]}}\lbs a\cdot z>t+l\delta-2A\rbs \Big),
	\end{align*}
	which implies the assertion of the proposition (notice that for $l=1$, we have an extra additive term of $1$; this is handled by replacing $\binom{k-1}{l}$ we obtained here with $\binom{k}{l}$ in the assertion of the proposition, as $\binom{k-1}{1}+1=\binom{k}{1}$).
\end{proof}

\subsection{Proof of Theorem~\ref{thm:large_wk}}
\label{sec:sub:proof-W^k}

Now we are ready to present the proof of Theorem~\ref{thm:large_wk}. Let us recall the statement of the theorem.

\medskip \noindent \textbf{Theorem~\ref{thm:large_wk}.}	There exist universal constants $c_1,c_2$ such that for any $k \in \mathbb{N}$ and for any halfspace $f_{t}=\one\{a\cdot x>t\}$ with $\eps=\mu(f_{t})<2^{-999k}$ and $\ii_{1}(f_{t})\leq c_1 \mu(f_{t})/k$, we have
	\[
	\wk{k}(f_{t})\geq \frac{(c_{2}\log(2k)) ^ {-k}}{k!} \eps^{2}\log(1/\eps)^{k}.
	\]

\medskip \noindent {\it Proof of Theorem~\ref{thm:large_wk}.}~ Let $f_t(x)$ be a halfspace that satisfies the assumptions of the theorem and denote $\eps=\mu(f)$.
Let $\beta$ be minimal such that $F(t+\beta)\leq\eps/3$ and let $\gamma$ be
minimal such that $\forall l\in \pintegers\colon F(t+l\gamma)\leq \eps/(6k)^{l}$. Denote	$\delta\ddd\beta+\gamma$. Note that by Theorem~\ref{thm:intro-strong_chernoff}, we have
    \begin{equation}\label{Eq:Aux-W^k-3}
    \delta\leq O \lbr \frac{\log (6k)}{\sqrt{\log(1/\eps)}} \rbr.
	\end{equation}
Let $A$ be the sum of the $k$ largest weights $a_{i}$, and let
$M\ddd\sum_{\lba S\rba =k}a^{S}e_{S}^{\delta}$, where $a^{S}=\prod_{i \in S}a_i$ and $e_{S}^{\delta}$ is as in Proposition~\ref{prop:e_s_lower_bound}. Similarly to the proof of Theorem~\ref{Thm:Main-W^1}, we are going to prove~\eqref{eq:wk_is_large} by combining upper and lower bounds on $M$.
We shall need two technical claims whose proofs will be presented in Section~\ref{sec:sub:auxiliary}.	
	\begin{claim}\label{claim:wk_small_A}
	For any $\eta>0$, there exists a constant $c=c(\eta)$ such that for any $k$ and for any halfspace $f_{t}$ with $\mu(f_{t})<2^{-999k}$ and $\ii_{1}(f_{t})\leq c \mu(f_{t})/k$, we have:
		\begin{itemize}
			\item $a_1 \leq \eta / \sqrt{k}$, and
			\item $2k a_1 < \beta$,
		\end{itemize}
where $a_1=\max_i a_i$ and $\beta$ is as defined above.
	\end{claim}
	
	\begin{claim}\label{claim:wk_coefs_decrease}
	Let $f_{t}$ be a halfspace that satisfies the assumptions of Theorem~\ref{thm:large_wk}. For any $x\in \spm^{n}$ with $f_{t}(x)=1$, we have $\sum_{\lba S\rba =k}a^{S}x^{S}\geq0$.
	\end{claim}
	
	As for all $s \geq 0$ and for any $x \in \{-1,1\}^n$ we have $f_{t}(x)\geq f_{t+s}(x) \geq 0$, Claim~\ref{claim:wk_coefs_decrease} implies that
\[
\forall s \geq 0, x \in \{-1,1\}^n\colon \sum_{\lba S\rba =k}a^{S}x^{S} f_{t}(x) \geq \sum_{\lba S\rba =k}a^{S}x^{S} f_{t+s}(x).
\]
Taking expectation over $x$ and using Fubini's theorem, we get that for any $s \geq 0$,
\[
\sum_{\lba S\rba =k} a^{S} \wh{f_{t}}(S) = \sum_{\lba S\rba =k} a^{S} \be_x \lbs x^{S} f_{t}(x) \rbs \geq \sum_{\lba S\rba =k}a^{S}
\be_x \lbs x^{S} f_{t+s}(x) \rbs = \sum_{\lba S\rba =k} a^{S} \widehat{f_{t+s}}(S).
\]
In particular, as each $e^{\delta}_S$ is a convex combination of expressions of the form $\widehat{f_{t+s}}(S)$, it follows that
\[
M = \sum_{\lba S\rba =k}a^{S}e_{S}^{\delta} \leq \sum_{\lba S\rba =k} a^{S} \wh{f_{t}}(S).
\]
By the Cauchy-Schwarz inequality, this implies
\begin{equation}\label{Eq:Aux-W^k-4}
M\leq\sqrt{\wk{k}(f_{t})}\sqrt{\sum_{\lba S\rba =k}\lbr a^{S}\rbr ^{2}}.
\end{equation}
On the other hand, by Proposition~\ref{prop:e_s_lower_bound} we have
\begin{equation}\label{Eq:Aux-W^k-5}
M \geq \frac{\sum_{\lba S\rba =k}\lbr a^{S}\rbr ^{2}}{\delta^{k}} \lbr \pr_z \lbs a\cdot z > t+2A \rbs - \sum_{l=1}^{\infty} k^{l}\pr_z \lbs a\cdot z > t+l\delta-2A \rbs \rbr.
\end{equation}
As $2A \leq 2ka_1 <\beta$ by Claim~\ref{claim:wk_small_A}, it follows from the definition of $\beta$ that $\pr_z \lbs a\cdot z > t+2A \rbs \geq \eps/3$. Using again the inequality $2A <\beta$, we have $t+l\delta-2A >t+l\delta-\beta \geq t+l\gamma$. Thus, by the definition of $\gamma$ we have $\pr_z \lbs a\cdot z > t+l\delta-2A \rbs \leq \eps/(6k)^l$. Substituting into~\eqref{Eq:Aux-W^k-5}, this yields
\begin{equation}\label{Eq:Aux-W^k-6}
M \geq \frac{\sum_{\lba S\rba =k}\lbr a^{S}\rbr ^{2}}{\delta^{k}} \lbr \frac{\eps}{3} - \sum_{l=1}^{\infty} k^l \frac{\eps}{(6k)^l} \rbr
\geq \frac{\eps}{9\delta^k} \sum_{\lba S\rba =k}\lbr a^{S}\rbr ^{2}.
\end{equation}
Combining Equations~\eqref{Eq:Aux-W^k-3},~\eqref{Eq:Aux-W^k-4} and~\eqref{Eq:Aux-W^k-6}, we obtain
	\begin{equation}\label{eq:wk_intermediate_lower_bound}
	\wk{k}(f_{t})\geq \frac{\eps^{2}}{81\delta^{2k}} \sum_{\lba S\rba =k}\lbr a^{S}\rbr ^{2} \geq \eps^{2} \frac{\log(1/\eps)^{k}}{O(\log (2k))^{k}} \sum_{\lba S\rba =k}\lbr a^{S}\rbr ^{2}.
	\end{equation}
Therefore, the assertion of the theorem will follow once we prove the following bound.	
	\begin{claim}\label{claim:wk_large_em}
	There exists $c>0$ such that for any $1 \leq m \leq k \in \mathbb{N}$ and for any halfspace $f_{t}(x)=\one\{a\cdot x > t\}$ with $\mu(f_{t})\leq 2^{-999k}$ and $\ii_{\max}(f_{t})\leq c \mu(f_{t})$, we have $\sum_{\lba S\rba =m}\lbr a^{S}\rbr ^{2}\geq 2^{1-m}/m!$.
	\end{claim}
	\begin{proof}[Proof of Claim~\ref{claim:wk_large_em}]
	Let $f_{t}$ be a function that satisfies the assumptions of the claim. Denote $b_{i}=a_{i}^{2}$, and notice $\sum b_{i}=1$. For each $m$, let $e_{m}=\sum_{\lba S\rba =m}b^{S}$, and $s_{m}=\sum_{i}b_{i}^{m}$. By the classical Newton-Girard formulas which relate elementary symmetric functions to sums-of-powers, we have
	\[
		me_{m}=\sum_{i=1}^{m}(-1)^{i-1}s_{i}e_{m-i}.
	\]
	Notice $e_{1}=s_{1}=1$. Furthermore, $s_{i}$ is a decreasing sequence satisfying
	\begin{equation}\label{Eq:Aux-W^k-7}
	s_{i+1}\leq s_i \max_j {b_j} \underbrace{\leq}_{\text{Claim }\ref{claim:wk_small_A}} \frac{s_i}{4k}.
	\end{equation}
		Now, we prove that $e_{m-1} \leq 2me_m$ for any $m\leq k$. Indeed,
		\bem
		me_{m}
		& \geq &
		\sum_{i=1}^{\lfloor m/2 \rfloor} \lbr s_{2i-1} e_{m+1-2i} - s_{2i} e_{m-2i}\rbr
		\\ & \underbrace{\geq}_{\text{\eqref{Eq:Aux-W^k-7}}} &
		\sum_{i=1}^{\lfloor m/2 \rfloor} \lbr s_{2i-1} e_{m+1-2i} - s_{2i-1} \frac{e_{m-2i}}{4k}\rbr
		\\ & \underbrace{\geq}_{\text{Induction}} &
		\sum_{i=1}^{\lfloor m/2 \rfloor} \lbr s_{2i-1} e_{m+1-2i} - \frac{2(m+1-2i)}{4k} s_{2i-1} e_{m+1-2i}\rbr
		\\ & \underbrace{\geq}_{m\leq k} &
		\frac{s_1}{2} e_{m-1} = \frac{e_{m-1}}{2}.
		\enm
    As $e_1=1$, the assertion follows by induction.
	\end{proof}
	
\noindent Equation~\eqref{eq:wk_intermediate_lower_bound} together with Claim~\ref{claim:wk_large_em} implies the desired inequality,
	\[
	\wk{k}(f_{t}) \geq \frac{O(\log(2k))^{-k}}{k!} \eps^2 \log(1/\eps)^k.
	\]
This completes the proof of Theorem~\ref{thm:large_wk}, modulo the proofs of Claims~\ref{claim:wk_small_A} and~\ref{claim:wk_coefs_decrease} that will be presented below.

\subsection{Proof of the Auxiliary Claims}
\label{sec:sub:auxiliary}

In this subsection we prove Claims~\ref{claim:wk_small_A} and~\ref{claim:wk_coefs_decrease}, thus accomplishing the proof of Theorem~\ref{thm:large_wk}.
\bigskip

\noindent \textbf{Claim~\ref{claim:wk_small_A}.} For any $\eta>0$, there exists a constant $c=c(\eta)$ such that for any $k$ and for any halfspace $f_{t}$ with $\mu(f_{t})<2^{-999k}$ and $\ii_{1}(f_{t})\leq c \mu(f_{t})/k$, we have:
		\begin{itemize}
			\item $a_1 \leq \eta / \sqrt{k}$, and
			\item $2k a_1 < \beta$,
		\end{itemize}
where $a_1=\max_i a_i$ and $\beta$ is as defined above.

\medskip

\begin{proof}
Let $f_{t}$ be a halfspace that satisfies $\mu(f_{t})<2^{-999k}$ and $\ii_{1}(f_{t})\leq c \mu(f_{t})$, with a sufficiently small $c$ to be determined below. (Note that the assumption on the influences of $f_{t}$ is weaker than the assumption of Claim~\ref{claim:wk_small_A}.) Denote $\eps=\mu(f_{t})$. By Theorem~\ref{Thm:Main-influence}, we have $c \eps \geq \ii_1(f_{t})\geq c' a_1 \eps \sqrt{\log(1/\eps)}$, and thus, $a_1 \leq c/c' \sqrt{\log(1/\eps)}$. As by assumption, $\eps \leq 2^{-999k}$, it follows that for a sufficiently small $c=c(\eta)$, we have $a_1 \leq \eta/\sqrt{k}$, as desired.
			
\medskip	

Now, we wish to show $2ka_{1}\leq \beta$, and for this we use the assumption: $\ii_1(f_{t})\leq c_1 \eps / k$, for a sufficiently small $c_1$. Consider $e_{1}^{\beta}$ defined in Lemma~\ref{lem:e_i_lower_bound}; explicitly, $e_{1}^{\beta}\ddd\be_{s\sim U(0,\beta)}\lbs \ii_{1}(f_{t+s})\rbs $.
Equation~\eqref{eq:e_i_trivial_lower_bound} of that lemma states that $e_{1}^{\beta}\geq\frac{a_{1}}{\beta}\Pr\lbs a\cdot x -a_{1} x_{1}\in\lbs t,t+\beta\rbs \rbs$.
From Corollary~\ref{cor:decay_of_influence} we deduce $e_{1}^{\beta}\leq5\ii_{1}(f_{t})$.
Hence, either
\begin{itemize}
\item $\Pr\lbs a\cdot x -a_{1} x_{1}\in\lbs t,t+\beta\rbs \rbs \leq \eps/3$, or

\item $\frac{a_1\eps}{3\beta} \leq e_1^{\beta} \leq 5\ii_1(f_{t})$.
\end{itemize}
    If $\frac{a_1\eps}{3\beta} \leq e_1^{\beta} \leq 5\ii_1(f_{t})$, then the assumption $\ii_1(f_{t})\leq c_1 \eps / k$ implies $\frac{a_1}{15\beta} \leq \frac{c_1}{k}$, and thus, $2ka_1 \leq \beta$, provided $c_1$ is sufficiently small.
    Thus, we may assume
    \begin{equation}\label{Eq:Aux-W^k-8}
    \Pr\lbs a\cdot x -a_{1} x_{1}\in\lbs t,t+\beta\rbs \rbs \leq \eps/3.
    \end{equation}
    Since for any $r,s$ we have
		\[
		\pr \lbs a\cdot x \in (r,s] \rbs \leq \pr \lbs a\cdot x-a_1 x_1 \in (r-a_1,s+a_1] \rbs,
		\]
	and as $\ii_{1}(f_{t}) = \Pr\lbs a\cdot x -a_{1} x_{1}\in (t-a_1,t+a_1] \rbs$ and similarly for $f_{t+\beta}$, it follows that
		\[
		\Pr\lbs a\cdot x -a_{1} x_{1}\in\lbs t,t+\beta\rbs \rbs +\ii_{1}(f_t)+\ii_{1}(f_{t+\beta}) \geq\Pr\lbs a\cdot x\in\lbr t,t+\beta\rbs \rbs \underbrace{\geq}_{\beta\text{-def.}} \frac{2\eps}{3}.
		\]
By~\eqref{Eq:Aux-W^k-8}, this implies
\begin{equation}\label{Eq:Aux-W^k-9}
\ii_{1}(f_{t})+\ii_{1}(f_{t+\beta}) \geq \frac{2\eps}{3}-\frac{\eps}{3}=\frac{\eps}{3}.
\end{equation}
However, as by Corollary~\ref{cor:decay_of_influence}, $\ii_{1}(f_{t+\beta})\leq 5\ii_{1}(f_{t})$,~\eqref{Eq:Aux-W^k-9} implies
\[
6\ii_{1}(f_{t})\geq \ii_{1}(f_{t})+\ii_{1}(f_{t+\beta}) \geq \frac{\eps}{3},
\]
which contradicts the assumption $\ii_1(f_{t}) \leq c_{1}\eps/k$ for a sufficiently small $c_1$. This completes the proof.
\end{proof}
	
\medskip \noindent In order to prove Claim~\ref{claim:wk_coefs_decrease}, we need another auxiliary claim.
	\begin{claim}\label{claim:wk_big_t}
		For any halfspace $f_{t}$ that satisfies the assumptions of Claim~\ref{claim:wk_small_A} with $\eta=1/16$, we have $t\geq 4\sqrt{k}$.
	\end{claim}
	The following proof method appears in~\cite{MON90}. For completeness we repeat it here.
	\begin{proof}[Proof of Claim~\ref{claim:wk_big_t}]
		We will show $\pr\lbs\sum a_i x_i > 4\sqrt{k}\rbs>2^{-999k}$. Since by assumption, $2^{-999k} \geq \mu(f_{t})=\pr[\sum a_i x_i > t]$, this will imply $t>4\sqrt{k}$.
		
Partition the $a_{i}$'s into sets $\{G_{s}\}_{s}$, each having sum-of-squares in $[\frac{1}{256k},\frac{1}{128k}]$; this is possible since $a_{i} \leq 1/(16\sqrt{k})$. (To be precise, at most one of the sets $G_s$ may have sum-of-squares less than $\frac{1}{256k}$. As will be apparent below, this does not affect the proof, so we neglect that set.) For each $s$, consider the random variable $X_s = (\sum_{i \in G_{s}} a_i x_i)^{2}$. It is easy to see that $\be[X_s] = \sum_{i \in G_s} a_i^2$, and that
\[
\be[X_s^2] = \be \lbs \lbr \sum_{i \in G_s} a_i x_i \rbr^4 \rbs = \sum_{i \in G_s} a_i^4 + 3 \sum_{i \neq j} a_i^2 a_j^2 \leq 3 \lbr \sum_{i \in G_s} a_i^2 \rbr^2 = 3 (\be[X_s])^2.
\]
Recall that the classical Paley-Zygmund inequality asserts that for any nonnegative random variable $Z$ with a finite second moment and for any $\alpha \in [0,1]$, we have
\[
\Pr \lbs Z> \alpha \be \lbs Z \rbs \rbs \geq \frac{(1-\alpha)^2 \be[Z]^2}{\mathrm{Var}[Z] + (1-\alpha)^2 \be[Z]^2}.
\]
Applying this inequality to the random variable $X_s = (\sum_{i \in G_{s}} a_i x_i)^{2}$, we get
	\[
		\pr \lbs \sum_{i \in G_{s}} a_i x_i > \lambda \sqrt{\sum_{i \in G_{s}} a_i^{2}} \rbs = \frac{1}{2} \pr \lbs X_s > \lambda^2 \be[X_s] \rbs \geq \frac{(1-\lambda^2)^2}{4+2(1-\lambda^2)^2}.
	\]
		As $1/(16\sqrt{k})\leq \sqrt{\sum_{i \in G_{s}} a_i^2}$ by the construction of $G_s$, we infer $\pr \lbs \sum_{i \in G_{s}} a_i x_i > 1/(32\sqrt{k}) \rbs > 1/10$. Since the number of sets $\{G_{s}\}_{s}$ is between $128k$ and $256k$, we obtain
\[
\pr\lbs\sum a_i x_i > 4\sqrt{k}\rbs \geq \pr \lbs \forall s\colon \sum_{i \in G_{s}} a_i x_i > 1/(32\sqrt{k}) \rbs > 10^{-256k} > 2^{-999k},
\]
as required.
\end{proof}
	
\noindent Now we are ready to prove Claim~\ref{claim:wk_coefs_decrease}.

\medskip \noindent \textbf{Claim~\ref{claim:wk_coefs_decrease}.} Let $f_{t}$ be a halfspace that satisfies the assumptions of Theorem~\ref{thm:large_wk}. For any $x\in \spm^{n}$ with $f_{t}(x)=1$, we have $\sum_{\lba S\rba =k}a^{S}x^{S}\geq0$.

\begin{proof}
		Let $f_{t}$ satisfy the assumptions of the claim, and let $x \in \spm ^ n$ be such that $f_{t}(x)=1$. Denote $b_{i}=b_i(x)=a_{i}x_{i}/\sum_{j\in [n]} a_{j} x_{j}$, so that $|b_{i}|\leq a_{i}/t$ for each $i$ and $\sum_{i} b_i = 1$. Define as before $e_{m}=\sum_{\lba S\rba =m}b^{S}$
		and $s_{m}=\sum_{i}b_{i}^{m}$. It is clear that for proving the claim, it is sufficient to prove $e_{k}\geq0$.
		
		Clearly, $\forall r\in\pintegers\colon \lba s_{2+r}\rba \leq s_{2}(\max b_i)^r$. As $\forall i\colon |b_{i}|\leq a_{i}/t$, this implies
        \begin{equation}\label{Eq:Aux-W^k-10}
        \forall r\in\pintegers\colon \lba s_{2+r}\rba \leq s_{2} \max\left\{ (a_{i}/t)^{r}\right\}.
        \end{equation}
        By Claims~\ref{claim:wk_small_A} and~\ref{claim:wk_big_t}, we have $\forall i\colon a_i \leq 1/16\sqrt{k}$ and $t \geq 4\sqrt{k}$ provided $c_1$ is sufficiently small, and hence,~\eqref{Eq:Aux-W^k-10} implies
        \[
        \forall r\in\pintegers\colon \lba s_{2+r}\rba \leq s_{2}\max\left\{ (a_{i}/t)^{r}\right\}\leq s_{2}/ (64k)^{r}.
        \]
		Similarly to Claim~\ref{claim:wk_large_em}, we shall prove by induction that $e_{m-1}\leq 2me_{m}$ for each $m \leq k$, and so in particular, $e_{k} \geq 2^{1-k}/k! > 0$, as required.

\mn From the Newton-Girard formulas, we have
		\bem
		\forall m \leq k\colon \qquad m e_{m}
		& = &
		s_{1} e_{m-1} + \sum_{i=2}^{m} (-1)^{i-1} s_{i} e_{m-i}
		\\
		& \geq &
		s_{1} e_{m-1} - \sum_{i=2}^{m} \frac{s_2}{(64k)^{i-2}} e_{m-i}
		\\
		& \underbrace{\geq}_{\text{Induction}} &
		s_{1} e_{m-1} - \sum_{i=2}^{m} \frac{2^{i-1}(m-1)^{i-1}}{(64k)^{i-2}} s_2 e_{m-1}
		\\
		& \geq &
		\lbr s_{1} - 3m s_{2} \rbr e_{m-1}
		\\
		& \underbrace{\geq}_{s_2 \leq 1/t^2} &
		(1-3k/t^{2}) e_{m-1}
		\\
		& \underbrace{\geq}_{\mrm{Claim}~\ref{claim:wk_big_t}}{} &
		e_{m-1}/2.
		\enm
This completes the proof.
	\end{proof}


\section{Noise Resistance and Correlation with a Halfspace}
\label{sec:noise}

Recall that a Boolean function $f$ is called \emph{Fourier noise resistant} if its first degree Fourier weight is within a constant factor of the maximal possible value, i.e., if $\wo(f)\geq c_0\mu(f)^2 \log(1/\mu(f))$ for a fixed constant $c_0$. In this section we prove Theorem~\ref{Thm:Main-Cor} which asserts that for any Boolean function $f$, there exists a halfspace $g$ such that $\cov(f,g)\geq\Omega\lbr \sqrt{\frac{\wo(f)}{\log(e/\wo(f))}}\rbr$. This implies that if $f$ is Fourier noise resistant then it is strongly correlated with some halfspace $g$.

In addition, we show that in the special case where $f$ is Fourier noise resistant, one can take the correlating halfspace $g$ to be unbiased, and also there exists a strongly biased halfspace $g''$ whose correlation with $f$ is `surprisingly large'. Finally, we prove Proposition~\ref{Prop:Prob-noise} which provides a `probabilistic' notion of noise sensitivity for biased functions that implies strong correlation with a halfspace.

\subsection{Proof of Theorem~\ref{Thm:Main-Cor} and a Tightness Example}

Let us recall the statement of the theorem.

\medskip \noindent \textbf{Theorem~\ref{Thm:Main-Cor}.} For any Boolean function $f$, there exists a halfspace $g$ such that
	\[
	\mathrm{Cov}(f,g) \geq c\sqrt{\wo(f) / \log(e/\wo(f))},
	\]
	where $c$ is an absolute constant.
In particular, if $f$ is Fourier noise resistant and $\mathbb{E}[f] \leq 1/2$ then there exists a halfspace $g$ such that $\mathrm{Cov}(f,g) = \Omega(\mathbb{E}[f])$.

\begin{proof}
	Let $f$ be a Boolean function, let $l(x) = \sum \hf(\{i\})x_i$ be the first Fourier level of $f$, and denote $a \ddd \lbn l\rbn_2=\sqrt{\wo(f)}$. Consider the family of biased halfpaces $\{g_t(x)=\one{\{l(x)>t\}}\}_{t \in \reals}$.
	The proof goes as follows: First, we show that the \emph{average} correlation of $f$ with a `random' $g_t$ is `not very small'. Then we use the Hoeffding inequality to assert that $\cov(f,g_t)$ is very small for a large $|t|$, and deduce that there exists $t$ such that $\cov(f,g_t)$ is `large', as asserted.
	
Define $h(t) \ddd \cov(f,g_t)$. We have
	\bem
		a^2 &=&
		\be_x\lbs f(x)l(x)\rbs \\
		&\underbrace{=}_{\text{Fubini}}&
		\int_{0}^{\infty}\be_{x}\lbs f(x)\cdot\one_{l(x)>t}\rbs \mrm{dt}-\int_{-\infty}^{0}\be_{x}\lbs f(x)\cdot\one_{l(x)<t}\rbs \mrm{dt}\\
		&=&
		\int_{0}^{\infty}\lbr \be_{x}\lbs f(x)\cdot\one_{l(x)>t}\rbs +\be_{x}\lbs f(x)\cdot\one_{l(x)\geq-t}\rbs -\be\lbs f\rbs \rbr \mrm{dt}\\
		& = &
		\int_{0}^{\infty}\lbr \be_{x}\lbs f(x)\cdot\one_{l(x)>t}\rbs - \be_{x} \lbs f(x) \rbs \be_{x} \lbs \one_{l(x)>t}\rbs \rbr + \\
		& & \qquad \qquad + \lbr \be_{x}\lbs f(x)\cdot\one_{l(x)\geq-t}\rbs - \be_{x} \lbs f(x) \rbs \be_{x} \lbs \one_{l(x) \geq -t}\rbs \rbr \mrm{dt}\\
        &=&
		\int_{0}^{\infty}h(t) \mrm{dt} + \int_{-\infty}^{0}h(t) \mrm{dt} = \int_{-\infty}^{\infty}h(t) \mrm{dt}.
	\enm
	By Hoeffding's inequality, for any $t>0$ we have $\mu\lbr g_{t}\rbr \leq \pr[l(x)>|t|]\leq\exp\lbr -t^{2}/2a^{2}\rbr $, and therefore, $\lba h(t)\rba \leq\exp\lbr -t^{2}/2a^{2}\rbr $ as well. Notice this also justifies the convergence of the above integrals. Let $r=\sqrt{6\log(2/a)}$. Then
	\[
	r\int_{ra}^{\infty}h(t)dt\leq\int_{ra}^{\infty}\frac{t}{a}h(t)\mrm{dt}\leq\int_{ra}^{\infty}\frac{t}{a}\exp\lbr -\frac{t^{2}}{2a^{2}}\rbr \mrm{dt}=a\cdot\exp\lbr -\frac{r^{2}}{2}\rbr \leq \frac{a^4}{8}.
	\]
	A symmetric argument implies $r\int_{-\infty}^{-ra}h(t)dt \leq \frac{a^4}{8}$, and hence,      $a^{2}-2\frac{a^4}{8r}\leq\int_{-ra}^{ra}h(t)\mrm{dt}$.
	Thus, there exists $t\in\lbr -ra,ra\rbr $
	with $h(t)\geq\frac{a^{2}-a^4/4r}{2ra}=\Omega\lbr \frac{a}{\sqrt{\log(2/a)}}\rbr $,
	as desired.
\end{proof}

Theorem~\ref{Thm:Main-Cor} is clearly tight (up to a constant factor) for any Fourier noise resistant function, as $\cov(f,g)$ cannot exceed $\mu(f)$. The following tightness example is of a different nature, being unbiased, monotone, and noise sensitive.

\begin{example}
Let $t(x)$ be the classical \emph{tribes} function defined by Ben-Or and Linial~\cite{BL90}. That is, we divide $[n]$ into tribes $T_1,T_2,\ldots,T_{n/r}$, each of size $r$, and let $t(x)=1 \Leftrightarrow \exists j\colon (x_i=1, \forall i \in T_j)$. The tribe size $r$ is chosen such that $\be[t(x)] \approx 1/2$. (The size is $r \approx \lg n - \lg \log n$.) One can easily show that for any halfspace $g(x)$, we have $\cov(t,g)=o_{n}(1)$.

Denote by $\eta$ the maximal correlation of $t(x)$ with a halfspace, so that $\eta=o(1)$. Let $h_{r}(x)=\one\{a\cdot x>r\}$ be a halfspace of measure $\eta$, and let $f(x)= t(x) \orr h_{r}(x)$. The function $f$ is monotone, we clearly have $\mu(f)\approx \frac{1}{2}$, and by Theorem~\ref{Thm:Main-W^1} we have
\[
\wo(f) \geq \Omega(\wo(h_{r})) = \Omega(\eta^2 \log(1/\eta)).
\]
On the other hand, as $f=t+h_{r}-t \cdot h_{r}$, for any halfspace $g$ we have
\[
\cov(f,g) = \cov(t,g)+\cov(h_{r},g) - \cov(t \cdot h_{r},g) \leq \cov(t,g) + \mu(h_{r}) + \mu(t \cdot h_{r}) \leq 3\eta,
\]
where the last inequality holds since $\mu(t \cdot h_{r}) \leq \mu(h_{r})=\eta$, and $\cov(t,g) \leq \eta$ by the definition of $\eta$. Therefore, the correlation of $f$ with any halfspace is at most $3\eta = O(\sqrt{\wo(f)/\log(e/\wo(f))})$, which means that Theorem~\ref{Thm:Main-Cor} is sharp for $f$.
\end{example}

\begin{remark}
A central feature of Theorem~\ref{Thm:Main-Cor} is that it holds also for non-monotone functions. In the monotone case, Theorem~\ref{Thm:Main-Cor} (together with the classical KKL theorem~\cite{KKL}) implies that for any unbiased monotone function $f$, there exists a halfspace $g$ such that $\mathrm{Cov}(f,g) = \Omega(\sqrt{\log n/n})$. A stronger (and optimal) result of $\Omega(\log n/\sqrt{n})$ was obtained by O'Donnell and Wimmer~\cite{OW13} who used their result to obtain a provably optimal weak learning algorithm for the class of monotone functions.

The result of O'Donnell and Wimmer also shows that Theorem~\ref{Thm:Main-Cor} is not tight for monotone unbiased functions with a `very small' $W^1$. Indeed, while the minimal possible value of $W^1$ is $\nu \sim (\log n)^2/n$ (attained by the tribes function), the result of~\cite{OW13} shows that maximal correlation with a halfspace for a monotone biased function is always at least $\sqrt{\nu}= \log(n)/\sqrt{n}$, and not $\sqrt{\nu}/{\sqrt{\log(e/\nu)}}$, as we would have obtained if Theorem~\ref{Thm:Main-Cor} was tight in that range.
\end{remark}

\subsection{A Stronger Correlation Theorem for Noise Resistant Functions}

Unlike the classical result of Benjamini et al.~\cite{BKS99} which states that any noise resistant function has a strong correlation with an \emph{unbiased} halfspace, Theorem~\ref{Thm:Main-Cor} does not guarantee that the correlating halfspace is unbiased. In the following proposition we show that in the special case where $f$ is noise resistant, one may require the correlating halfspace to be unbiased, like in~\cite{BKS99}.

\begin{proposition}\label{prop:correlation_with_unbiased_ltf}
	For any Fourier noise resistant Boolean function $\lfunc{f}$, there exists an unbiased
	halfspace $g_{0}$ with $\cov(f,g_{0}) \geq\Omega(\mu(f))$.
\end{proposition}

In the proof of the proposition we use the classical \emph{noise operator} $T_{\rho}$, which lies behind the notion of noise sensitivity. The noise operator is defined as $T_{\rho}f(x) \ddd  \be[f(y)]$, where $y$ is obtained from $x$ by independently keeping each coordinate of $x$ unchanged with probability $\rho$, and replacing it by a random value with probability $1-\rho$. It has a convenient representation in terms of the Fourier expansion of $f$: we have $T_{\rho}(f) = \sum_S \rho^{|S|} \wh{f}(S)$, and thus, by the Parseval identity, $\be [f \cdot T_{\rho} g] = \sum_S \rho^{|S|} \wh{f}(S) \wh{g}(S)$, for any $f,g$. (Note that the noise stability $\mathbb{S}_{\rho}(f)$ is simply $\be [f \cdot T_{\rho} f]$.) The method we use in the proof was introduced in~\cite{KKM16}.

\begin{proof}
    Let $f\colon\{-1,1\}^n \rightarrow \{0,1\}$ be a Fourier noise resistant function. Denote $\mu(f)=\eps$, so that $\wo(f)\geq c_0\eps^{2}\log\left(1/\eps\right)$. Let $l(x)=f^{=1}(x) \ddd \sum_{i=1}^{n} \wh{f}(\{i\})x_i$, and denote  $g_{0}(x)=\sgn(l(x))$. We show that for an appropriate choice of $\rho$, the function $f$ has a strong correlation with the `noisy version' $T_{\rho}g_{0}$. As $\tr g_{0}$ is a convex combination of unbiased halfspaces, this will imply that there exists an unbiased halfspace $g'_{0}$ that strongly correlates with $f$.

    Let $\rho$ be a parameter to be chosen below. Since $\be[T_{\rho}g_{0}]=\be[g_{0}]=0$, we have
	\begin{equation}\label{Eq:Aux-Noise1}
	\cov(f,T_{\rho}g_{0}) = \be\left[f\cdot T_{\rho}g_{0}\right]=\sum_{S}\rho^{\left|S\right|}\hf(S)\widehat{g_{0}}(S)
	\geq \rho \sum_{i=1}^n \wh{f}(\{i\}) \wh{g_{0}}(\{i\})  -\sum_{k\geq2}\rho^{k}\sqrt{W^{k}(f)},
	\end{equation}
	where the last inequality uses Cauchy-Schwarz. As $l(x)=f^{=1}(x)$, we have
    \begin{equation}\label{Eq:Aux-Noise2}
    \sum_{i=1}^n \wh{f}(\{i\}) \wh{g_{0}}(\{i\}) = \be[l(x) g_{0}(x)]
    = \be[l(x) \sgn(l(x))]
    = \be[|l(x)|] = \lbn l \rbn_1 \geq \normts{l}/\sqrt{2},
    \end{equation}
    where the first equality uses Parseval's identity and the last inequality employs the Khintchine-Kahane inequality. As $f$ is Fourier noise resistant, we have $\lbn l \rbn_2 =\sqrt{\wo(f)} \geq \sqrt{c_0\eps^{2}\log(1/\eps)}$, and so combining~\eqref{Eq:Aux-Noise1} with~\eqref{Eq:Aux-Noise2} we get
    \begin{equation}
    \cov(f,T_{\rho}g_{0}) \geq  \rho\cdot\frac{\sqrt{c_0\eps^{2}\log(1/\eps)}}{\sqrt{2}}-\sum_{k\geq 2}\rho^{k}\sqrt{W^{k}(f)}.
    \end{equation}
	Using the level-$k$ Inequality (Theorem~\ref{Thm:Level-k} above) which asserts that
	\begin{equation}\label{Eq:Aux-noise2.7}
	\forall k\leq 2\log(1/\mu(f)):\qquad W^{k}(f)\leq\left(\frac{2e}{k}\log(1/\eps)\right)^{k}\eps^{2},
	\end{equation}
	we obtain
	\[
	\be\left[f\cdot T_{\rho}g_0\right]\geq\rho\eps\cdot\frac{\sqrt{c_0\log(1/\eps)}}{\sqrt{2}}-\sum_{k=2}^{2\log(1/\eps)}\rho^{k}\eps\cdot\left(\frac{2e}{k}\log(1/\eps)\right)^{k/2}-\sum_{k > 2\log(1/\eps)} \rho^{k}.
	\]
	Taking $\rho= (1/2e)^{2}\sqrt{c_0/\log(1/\eps)}$, and noting that w.l.o.g. we may assume $c_0\leq1$, results in
	\[
	\be\left[f\cdot T_{\rho}g_{0}\right]\geq\frac{c_0\eps}{\sqrt{32}e^{2}}-c_{0}\eps\sum_{k \geq 2}\left(2e\right)^{-3k/2} - c_{0}\eps\sum_{k>2\log(1/\eps)}(2e)^{-3k/2} \geq\Omega(c_0\eps).
	\]
	Finally, note that by the definition of the noise operator, the function $T_{\rho} g_{0}$ is a convex combination of unbiased halfspaces of the form $\sgn\lbr\sum_{i=1}^n (-1)^{\alpha_i} \wh{g_{0}}(\{i\}) x_i \rbr$, where $\forall i\colon \alpha_i \in \{0,1\}$. Hence, there exists an unbiased halfspace $g'_{0}$ such that $\cov(f,g'_{0}) =\Omega(c_0\epsilon)$, as asserted.
\end{proof}

Interestingly, one cannot guarantee that the linear form associated with the correlating halfspace is simply $l=f^{=1}$ like in the unbiased case studied in~\cite{BKS99}, as can be seen in the following example.
\begin{example}
Let $f\colon\{-1,1\}^5 \rightarrow \{0,1\}$ be defined as $f(x)=1 \Leftrightarrow \sum_{i=1}^5 x_i \in \{-1,3,5\}$. We have $\wh{f}(\{i\})=1/16$ for all $1 \leq i \leq 5$, and so, the linear form $f^{=1}$ defines the majority function $g_{0}(x)=\one\{\sum_{i=1}^{5} x_i > 0\}$. However, a direct computation shows that $\cov(f,g_{0})=-1/16$ is not even positive! On the other hand, if we reverse the sign of one variable, i.e., consider $g'_{0}=\one\{x_1+x_2+x_3+x_4-x_5>0\}$, we obtain $\cov(f,g'_{0})=1/8$.
\end{example}
\noindent Similar examples can be constructed for large values of $n$ as well.

\subsection{Any Fourier Noise Resistant Function Correlates Well with a Biased Halfspace}

We now present another proposition which shows that any Fourier noise resistant function correlates well with a strongly biased halfspace. This result is somewhat surprising, as biased functions correlate badly in general.
\begin{proposition}
	For any function $\lfunc{f}$ that satisfies $\eps = \mu(f)$ and $\wo(f)\geq \alpha\eps^{2}\log\left(\frac{1}{\eps}\right)$ for some $\alpha = \omega\left(\eps^{2}\right)$, there exists a halfspace $g_{s}:\spm^{n}\to\zo$ such that
	$\be\left[fg\right]\geq\Omega(\sqrt{\alpha}\eps)$ and $\be\left[g_{s}\right]\le\eps^{\alpha/8}$.
	In particular, if $\alpha = \Omega(1)$ then $\cov(f,g_{s})\geq \Omega(\sqrt{\alpha}\eps)-O(\eps^{1+\alpha/8})=\Omega(\eps)$.
\end{proposition}
\begin{proof}
	Let $s=\frac{1}{2}\sqrt{\alpha\log(1/\eps)}$, and consider the halfspace $g_{s}(x)=\one{\{l(x)>s\}}$ (with $\normts{l}=1$). By Hoeffding's
	inequality, we have $\be[g_{s}]\leq\eps^{\alpha/8}$. Hence, in order to prove the proposition we have to show that
\begin{equation}\label{Eq:Aux-Noise3.5}
\be\left[f g_{s}\right]\geq \Omega(\sqrt{\alpha}\eps).
\end{equation}
    Denote $l=\frac{1}{\left\Vert f^{=1}\right\Vert _{2}}f^{=1}$. We clearly have
	\begin{equation}\label{Eq:Aux-Noise4}
	\sqrt{\wo(f)}=\left\langle f,l\right\rangle =\be\left[\one\{l(x)\leq s\}f(x)l(x)\right]+\be\left[\one\{l(x)>s\}f(x)l(x)\right].
	\end{equation}
	Since
	\[
	\be\left[\one\{l(x)\leq s\}f(x)l(x)\right]\leq s\eps\leq\frac{1}{2}\sqrt{\wo(f)},
	\]
	it follows from~\eqref{Eq:Aux-Noise4} that
	\[
	\be\left[\one\{l(x)>s\}l(x)f(x)\right]\geq\frac{1}{2}\sqrt{\wo(f)}.
	\]
	Note that for any $u\geq s$,
	\[
	\be\left[g_{s}(x)l(x)f(x)\right]\leq u\be\left[g_{s}(x) f(x)\right]+\int_{u}^{\infty}\Pr\left[l(x)>t\right]\mbox{dt} =
    u\be[f g_{s}]+ \int_{u}^{\infty}\Pr\left[l(x)>t\right]\mbox{dt}.
	\]
	Since $\left\Vert l\right\Vert _{2}=1$, Hoeffding's inequality yields
	$\Pr\left[l(x)>t\right]\leq\exp\left(-t^{2}/2\right)$. Thus,
	\[
	\int_{u}^{\infty}\Pr\left[l(x)>t\right]\mbox{dt}\leq\int_{u}^{\infty}\frac{t}{u}\Pr\left[l(x)>t\right]\mbox{dt}\leq\int_{u}^{\infty}
    \frac{t}{u}\exp\left(-t^{2}/2\right)\mbox{dt}=\frac{1}{u}\exp\left(-u^{2}/2\right).
	\]
	Taking $u=2\sqrt{\log(1/\eps)}$ and combining the four previous inequalities, we obtain
	\begin{equation}\label{Eq:Aux-Noise5}
	s\eps\leq\frac{1}{2}\sqrt{\wo(f)}\leq u\be [f g_{s}]+\frac{1}{u}\eps^{2}.
	\end{equation}
	Finally, note that we have $\frac{1}{u}\eps^{2}\leq\frac{1}{2}s\eps$, as otherwise we have $\sqrt{\alpha}\log(1/\eps)=u s\leq2\eps$, and hence, $\alpha\leq4\eps^{2}$, which contradicts the assumption. Therefore,~\eqref{Eq:Aux-Noise5} gives
	\[
	\frac{\sqrt{\alpha}}{8}\eps\leq\frac{1}{2}\frac{s}{u}\eps\leq\be [f g_{s}].
	\]
    This proves~\eqref{Eq:Aux-Noise3.5}, thus completing the proof of the proposition.
\end{proof}

\subsection{A Probabilistic Notion of Noise Resistance}

We conclude this section with a `probabilistic' notion of noise resistance that implies strong correlation with a halfspace. Recall that by the definition of~\cite{BKS99}, a function $f$ is called noise resistant if $\mathbb{S}_{\rho}(f)=\Omega(1)$ for any constant $\rho$. If we want to generalize this definition to biased functions, the rate of the noise we consider must depend on the expectation of the function, as will be shown below.
In order to find a natural rate of noise for biased functions, we use (once again) the relation of noise sensitivity to the Fourier expansion of the function.

\paragraph{Determining the `right' rate of noise.} Using the expansion
\[
\mathbb{S}_{\rho}(f)=\be [f \cdot T_{\rho} f] = \sum_S \rho^{|S|} \wh{f}(S)^2
\]
and the level-$k$ inequality (i.e., Theorem~\ref{Thm:Level-k} above), we get
\[
\mathbb{S}_{\rho}(f) = \sum_k \rho^k W^{k}(f)\leq \sum_k \left(\frac{2e \rho}{k}\log(1/\eps)\right)^{k}\eps^{2}+\sum_{k>2\log(1/\eps)} \rho^{k},
\]
where $\eps=\mu(f)$. By Stirling's approximation, $n! \approx \sqrt{2\pi n}(n/e)^n$,
this implies
\[
\mathbb{S}_{\rho}(f)-\frac{\rho^{2\log(1/\eps)}}{1-\rho} \leq
\eps^2 \sum_k \frac{(c' \rho \log(1/\eps))^k}{k!} \leq \eps^2 \exp(c' \rho \log(1/\eps)),
\]
where $c'$ is a universal constant. It follows that if $\rho=o(1/\log(1/\eps))$ then $\mathbb{S}_{\rho}(f)$ is very small for any function $f$. Hence, we consider noise rate of $\rho=\Theta(1/\log(1/\eps))$, for which $\mathbb{S}_{\rho}(f)$ can be as large as $\mu(f)^2$, and say that $f$ is \emph{noise resistant} if $\mathbb{S}_{\rho}(f)=\Omega(\mu(f)^2)$.

We we prove Proposition~\ref{Prop:Prob-noise} which asserts that this notion of noise resistance implies strong correlation with a halfspace. 

\medskip \noindent \textbf{Proposition~\ref{Prop:Prob-noise}.} There exists a universal constant $c > 0$ such that for any monotone function $f\colon \spm^n \rightarrow \zo$ with $\mathbb{E}[f] \leq \frac{1}{2}$, if $\mathbb{S}_{c/\log(1/\mathbb{E}[f])}(f) = \Omega(\mathbb{E}[f]^2)$, then $W^1(f) =\Omega(\mathbb{E}[f]^2 \log(1/\mathbb{E}[f]))$, and consequently, there exists a halfspace $g$ such that $\mathrm{Cov}(f,g) = \Omega(\mathbb{E}[f])$.

\medskip

\begin{proof}
	Recall that the quantitative version of the BKS noise sensitivity theorem~\cite{KK10} asserts that for any monotone $f$ and for any $k$ which satisfies $\wo(f)\leq \exp(-2(k-1))$, we have
	\begin{equation}\label{Eq:Aux-Noise6}
	\wk{k}(f)=\sum_{|S|=k} \hf(S)^2 \leq \frac{5e}{k} \wo(f) \lbr \frac{2e\log(k/\wo(f))}{k-1} \rbr ^ {k-1}.
	\end{equation}
	Let $f$ be a function that satisfies the assumptions of the proposition. Denote $\eps = \mu(f)$ and let $\alpha$ satisfy $\wo(f)=\eps^{\alpha}$. We shall compute an upper bound for $\sss(f)$ when $\rho=c/\log(1/\eps)$, for a constant $c$ to be specified below. Set $T=\log(1/\eps)/2$, and note that since $\wo(f) \leq \epsilon$ by the Poincar\'{e} inequality, we may apply~\eqref{Eq:Aux-Noise6} for all $k \leq T$. Hence, we have
	\begin{align}\label{Eq:Aux-Noise7}
    \begin{split}
	\sss(f) = \sum_{k=1}^{n} \rho^k \wk{k}(f)
	& \leq
	5e\sum_{k=1}^{T} \wo(f)\lbr 6\log(k/\wo(f)) \rbr ^ {k-1} \rho ^ k + \sum_{k=T}^{\infty} \rho^k \\
    & \leq
	5e\wo(f)\rho \sum_{k=1}^{T} \lbr 6c\frac{(\alpha+1)\log(1/\eps)}{\log(1/\eps)} \rbr ^ {k-1} + \sum_{k=T}^{\infty} c^k \\
	& \leq
	O \left( \frac{\wo(f)}{\log(1/\eps)} \lbr 1 + \lbr \frac{\alpha+1}{100} \rbr ^ {\log(1/\eps)}\rbr \right) + O(\eps^{3}),
    \end{split}
    \end{align}	
	where the last inequality follows by taking $c$ to be a sufficiently small constant and the `$100$' in the denominator can be taken to be
    any constant (determined by $c$).

    For $\alpha \leq 99$,~\eqref{Eq:Aux-Noise7} implies $\sss(f)=O\lbr \frac{\wo(f)}{\log(1/\eps)}+\eps^3\rbr$. As we have $\sss(f)=\Omega(\eps^2)$ by assumption, this implies $\wo(f)\geq \Omega(\eps^2\log(1/\eps))$, as asserted.

    For $\alpha>99$,~\eqref{Eq:Aux-Noise7} yields $\eps^{2-\alpha} \leq O \lbr \lbr \frac{\alpha+1}{100}\rbr^{\log(1/\eps)}\rbr= O\lbr \eps^{-\log((\alpha+1)/100)}\rbr$, which can not happen for $\eps$ small enough (note that we may assume $\eps$ is small, as we control $c$).
    This completes the proof.
\end{proof}

\begin{remark} In~\cite[Theorem 3.10.4]{OD13_phd}, O'Donnell presented another notion of noise resistance that is satisfied by biased halfspaces. He showed that for any halfspace $g_{t}$ with $\eps = \mu(g_{t})$,
	\begin{equation}\label{eq:biased_peres_theorem}
	\mrm{NS}_{\delta/\log(1/\eps)}(g_{t})=O(\eps\sqrt{\delta}),
	\end{equation}
    where $\mrm{NS}_{\eta}(g_{t})=\frac{1}{2}-\frac{1}{2}\mathbb{S}_{1-2\eta}(g_{t})$. (This provides a biased version of the noise stability theorem of Peres for halfspaces~\cite{Peres04}).	In light of the previous results, one might wonder whether every monotone function $f$ that satisfies~\eqref{eq:biased_peres_theorem} is well-correlated with some halfspace. This indeed holds for unbiased functions, as in this case, ~\eqref{eq:biased_peres_theorem} implies that $f$ is noise stable (according to the notation of~\cite{BKS99}), and consequently, satisfies $W^1(f)=\Omega(1)$, which in turn implies that $f$ correlates well with a halfspace.

    However, this does not generalize to the biased setting, as can be seen in the following example. Let
	\[
	f(x)=\Orr_{i\in [a]} \Andd_{j\in [b]} x_{i,j}
	\]
	be a variant of the tribes function, with $a=1/\eps$ and $b=2\lg(1/\eps)$. Clearly, $\mu(f)\approx \eps$. It can be shown that on the one hand, $f$ satisfies~\eqref{eq:biased_peres_theorem}, and on the other hand, $f$ does not correlate well with any halfspace (i.e., $\cov(f,g_{t})=o(\eps)$ for any halfspace $g_{t}$).
\end{remark}



\section*{Acknowledgments} 

We are grateful to Gil Kalai and to Ryan O'Donnell for numerous useful suggestions,
to G\'{a}bor Lugosi for communicating to us the paper~\cite{DL08},
and to Rani Hod for suggesting Example~\ref{ex:vb_large_ratio}.


\bibliographystyle{amsplain}


\begin{dajauthors}
\begin{authorinfo}[nk]
  Nathan Keller\\
  Bar Ilan University\\
  Ramat Gan, Israel\\
  nathan.keller27\imageat{}gmail\imagedot{}com \\
  \url{http://u.math.biu.ac.il/~nkeller/}
\end{authorinfo}
\begin{authorinfo}[johan]
  Ohad Klein\\
  Bar Ilan University\\
  Ramat Gan, Israel\\
  ohadkel\imageat{}gmail\imagedot{}com \\
\end{authorinfo}
\end{dajauthors}

\end{document}